\documentclass{amsart}
\usepackage{graphicx}
\usepackage{amsmath}
\usepackage{amscd}
\usepackage{mathtools}
\usepackage{orcidlink}
\usepackage{enumerate}
\usepackage{tikz-cd}
\usepackage[all,cmtip]{xy}
\usepackage{tikz}
\usepackage{amssymb}
\usepackage{ascmac} 
\usepackage{color}
\usepackage{autobreak}
\usepackage{breqn}
\newtheorem{dfn}{Definition}[section]
\newtheorem{thm}[dfn]{Theorem}
\newtheorem{pro}[dfn]{Proposition}
\newtheorem{lem}[dfn]{Lemma}

\theoremstyle{definition}
\newtheorem{exa}[dfn]{Example}

\newcommand{\mb}{\mathbb}
\newcommand{\mr}{\mathrm}

\newcommand{\mac}{\mathcal}

\title[Automorphisms with Fixed Loci of Codimension at Most Two]{Automorphisms of Smooth Hypersurfaces with Fixed Loci of Codimension at Most Two}
\author{Taro Hayashi}
\author{Ryoichi Suzuki\orcidlink{0000-0001-9979-1882}}
\address{(Taro Hayashi)
Department of Mathematical Sciences,
Ritsumeikan University,
1$-$1$-$1 Nojihigashi, Kusatsu, Shiga, 525$-$8577, Japan}
\email{haya4taro@gmail.com}
\address{(Ryoichi Suzuki)
Department of Business Economics, School of Management, Tokyo University of Science, 1$-$11$-$2 Fujimi, Chiyoda-ku, Tokyo 102$-$0071, Japan}
\email{rsuzukimath@gmail.com}

\date{\today}
\subjclass{Primary 14J50; Secondary 14L30, 14E20}
\keywords{Projective hypersurface; Automorphisms; Fixed points set}
\begin{document}
\maketitle
\begin{abstract}
We study automorphisms of smooth hypersurfaces in projective space $\mathbb{P}^{n+1}$ whose fixed loci have codimension at most two for $n\geq2$. While classifications of possible orders of automorphisms are known, our aim is to explore the relationship between the order of an automorphism and its algebraic and geometric properties. In this paper, we show that the assumption on the fixed locus restricts the possible orders of automorphisms. 
Moreover, when the fixed locus has codimension at most two, we investigate the rationality of quotient spaces associated with automorphisms whose orders are multiples of $d-1$ or $d$, where $d$ denotes the degree of the hypersurface.
\end{abstract}
\section{Introduction}
Throughout this paper, we work over an algebraically closed field $k$ of characteristic zero.
Automorphisms of algebraic varieties reflect intrinsic geometric symmetries and play a central role in algebraic geometry.
For smooth hypersurfaces in projective space, it is known that the automorphism group is finite and, under mild assumptions on the dimension and the degree, consists of projective linear transformations ([\ref{bio:acgh},\ref{bio:mm63}]).
However, beyond these general structural results, explicit descriptions of automorphism groups are available only in a limited number of special cases ([\ref{bio:bb16a},\ref{bio:bc25},\ref{bio:ou19},\ref{bio:wu20},\ref{bio:yz24}]).

An intermediate approach between abstract structural results and complete group-theoretic classifications is the study of the order of automorphisms.
The order provides a refined invariant.
Even when a description of the automorphism group is unavailable, it still allows us to detect symmetry phenomena.
For instance, this phenomenon is studied in [\ref{bio:bb16n},\ref{bio:th21o}].
Most existing studies on automorphism orders of hypersurfaces focus primarily on numerical constraints arising from the degree or the dimension
([\ref{bio:bb16n},\ref{bio:gl11},\ref{bio:gl13},\ref{bio:z22}]).
In particular, the possible orders of automorphisms are determined in
[\ref{bio:bb16n},\ref{bio:z22}]. When $n=1$, 
the orders are known very explicitly in [\ref{bio:bb16n}]. 
Thus, the orders themselves are already known from a numerical viewpoint, and throughout this paper we restrict our attention to the case $n \geq 2$.

In contrast to the existing literature, we incorporate geometric information into the study of automorphism orders by imposing conditions on the dimension of the fixed locus. 
More precisely, we consider automorphisms of smooth projective hypersurfaces whose fixed loci have small codimension. 
The fixed-locus condition restricts them and yields a more refined classification. A related viewpoint appears in~[\ref{bio:th21s},\ref{bio:th25a}], which treats finite automorphism groups acting freely. In particular, when the fixed locus is empty, the order of the automorphism divides the degree of the hypersurface.
Our first main result shows that, under the assumption that the fixed locus has codimension at most two, the possible orders are strongly constrained and can be explicitly determined.
\begin{thm}\label{1.1}
Let $\mac X \subset \mathbb P^{n+1}$ be a smooth hypersurface of degree $d$ for $n\geq2$, 
and let $g\in \mathrm{PGL}(n+2,k)$ be an automorphism of $\mac X$.

We assume that the fixed locus $\mathrm{Fix}(g)$ has codimension one in $\mathcal X$.
Then the order of $g$ divides
\[
d,\quad d-1,\quad \text{or}\quad d-2.
\]
Moreover, if the order of $g$ is at least $3$ and divides $d-2$, then $n=2$.

We assume that the fixed locus $\mathrm{Fix}(g)$ has codimension two in $\mathcal X$.
Then the following statements hold.
\begin{enumerate}[$(i)$]
\item If $n\geq 4$, then the order of $g$ divides
\[
(d-1)d,\quad (d-1)^2,\quad\mr{or}\quad (d-2)d.
\]
\item If $n=3$, then the order of $g$ divides
\[
(d-1)d,\quad (d-1)^2,\quad (d-2)d,\quad d^2-3d+3,\quad \mr{or}\quad (d-2)(d-1).
\]
\item If $n=2$, then the order of $g$ divides
\[
(d-1)^2d,\quad (d-1)^3,\quad (d^2-3d+3)d,\quad (d^2-3d+3)(d-1),
\]
\[
(d-2)(d-1)d,\quad \mr{or}\quad (d-2)(d-1)^2.
\]
\end{enumerate}
\end{thm}
One geometric aspect of an automorphism is the structure of the quotient variety induced by its action.
In particular, automorphisms whose quotients are rational varieties
are studied in several works $[\ref{bio:bb16n},\ref{bio:th21l},\ref{bio:th23g},\ref{bio:th25g},\ref{bio:th25a},\ref{bio:hs25l}]$,
where various sufficient conditions guaranteeing the rationality of the quotient space are given.
In this paper, we investigate this problem for automorphisms of smooth hypersurfaces whose fixed loci have codimension at most two.
Our second main result focuses on automorphisms whose orders are divisible by $d-2$, $d-1$, or $d$, where $d$ denotes the degree of the hypersurface, and provides sufficient conditions for the rationality of the corresponding quotient varieties.
The restriction to automorphisms of such orders is motivated by the theory of Galois points ([\ref{bio:my00},\ref{bio:y03}]) and by the notion of Galois skew lines $([\ref{bio:th25g}])$.
The problem of detecting the existence and the position of Galois points from automorphisms are studied in [\ref{bio:bb16n},\ref{bio:th21l}]:
the former treats the case $n=1$ under the assumption $\operatorname{ord}(g)=k(d-1)$ or $kd$ with $k\ge2$, while the latter considers arbitrary $n$ assuming that the fixed locus of the automorphism has codimension at most two.
\begin{thm}\label{1.2}
Let $\mathcal X \subset \mathbb P^{n+1}$ be a smooth hypersurface of degree $d$ for $n\ge 2$, and let $g\in \mathrm{PGL}(n+2,k)$ be an automorphism of $\mathcal X$.
We have the following:
\begin{enumerate}[$(i)$]
\item We assume that $\operatorname{ord}(g)\geq3$ and $\mr{ord}(g)\in\{d-2, d-1,d\}$.
If $\operatorname{Fix}(g)$ has codimension one in $\mathcal X$,
then the quotient space $\mathcal X/\langle g\rangle$ is a rational variety. 
\item We assume that $\operatorname{ord}(g)=k(d-1)$ or $kd$ for some integer $k \ge 2$.
If $\operatorname{Fix}(g)$ has codimension one in $\mathcal X$,
then the quotient space $\mathcal X/\langle g\rangle$ is a rational variety. 
\end{enumerate}
Here $\langle g\rangle$ is a cyclic group which is generated by $g$.
\end{thm}

\section{Preliminaries}
In this section, we recall basic notions and known results on automorphisms of hypersurfaces that are relevant to this paper.
We fix notation and recall several elementary facts that will be used repeatedly in the proofs of our main results.

Let $\mathcal{X} \subset \mathbb{P}^{n+1}$ be a smooth hypersurface of degree $d$.
Suppose that $n = 1$ with $d \geq 4$, or $n \geq 2$ with $d \geq 3$, excluding the case $(n,d) = (2,4)$.
Then every automorphism of $\mathcal{X}$ is induced by an element of $\mathrm{PGL}(n+2,k)$ $([\ref{bio:acgh},\ \mathrm{Appendix\ A,\ Exercise}\ 18]$, $[\ref{bio:mm63},\ \mathrm{Theorem\ 1\ and\ Theorem\ 2}])$.

For smooth plane curves in $\mathbb P^2$ of degree at least four,
the maximal orders of automorphisms are determined by Badr and Bars.
\begin{thm}\label{bb1}$([\ref{bio:bb16n},\ {\rm Theorem}\ 1\ {\rm and}\ {\rm Theorem}\ 6])$.
Let $X\subset \mathbb P^2$ be a smooth plane curve of degree $d\geq 4$, and let $f$ be an automorphism of $X$.
Then the order of $f$ divides
\[
(d-1)d,\quad (d-1)^2,\quad  (d-2)d,\quad  \mr{or}\quad d^2-3d+3.
\]
\end{thm}
Moreover, Badr and Bars show that, for each extremal case, a smooth plane curve admitting an automorphism of the corresponding maximal order is uniquely determined up to isomorphism. They also determine the full automorphism group of each such curve.
In higher dimensions, Zheng determined the maixaml possible orders of linear automorphisms of smooth hypersurfaces.
\begin{thm}$([\ref{bio:z22},\ \mr{Theorem}\ 4.8])$.
Let $d\ge 3$ and $n\ge 1$ be integers.
Let $k$ be the order of a linear automorphism of a smooth hypersurface of degree $d$
and dimension $n$ in $\mathbb P^{n+1}$.
Then $k$ is a factor of one of the following integers:
\begin{enumerate}[$(i)$]
\item$\frac{|1-(1-d)^{n+2}|}{d}$;
\item$(d-1)^{n+1}$;
\item$|1-(1-d)^a| \quad \text{for } 1\le a\le n+1$;
\item$\mathrm{lcm}\bigl(|1-(1-d)^{a_1}|,\dots,|1-(1-d)^{a_t}|\bigr)$,
where $t\ge 2$, $1\le a_1<\cdots<a_t$, and $\sum_{i=1}^t a_i\le n+2$;
\item$\mathrm{lcm}\bigl(|1-(1-d)^{a_1}|,\dots,|1-(1-d)^{a_t}|,(d-1)^{\,b-1}\bigr)$,
where $t\ge 1$, $1\le a_1<\cdots<a_t$, $b\ge 2$, and $\sum_{i=1}^t a_i+b\le n+2$.
\end{enumerate}
\end{thm}
Zheng’s theorem gives a complete list of the maximal orders from a numerical viewpoint.
The aim of the present paper is to refine this classification by incorporating geometric information, namely the structure of fixed loci.

We fix homogeneous coordinates $[X_0:\cdots:X_{n+1}]$ on the projective space $\mathbb{P}^{n+1}$.
For integers $0 \le r \le n+1$,
we set
\[
W(X_0,\ldots,X_l)
:=\{[X_0:\cdots:X_{n+1}]\in\mathbb P^{n+1}\mid X_0=\cdots=X_l=0\}.
\]
Let $P^r, P^{n-r} \subset \mathbb{P}^{n+1}$ be projective subspaces of dimensions $\dim P^r = r$ and $\dim P^{n-r} = n-r$, respectively, such that
$P^r \cap P^{n-r} = \emptyset$.
After a suitable change of homogeneous coordinates on $\mathbb{P}^{n+1}$, we may assume that 
$P^r = W(X_{n-r+1}, \ldots, X_{n+1})$ and 
$P^{n-r} = W(X_{0}, \ldots, X_{n-r})$.
Consider the projection map 
\begin{align*}
\mathbb{P}^{n+1} \setminus (P^r \cup P^{n-r})
&\longrightarrow \mathbb{P}^r \times \mathbb{P}^{n-r}, \\
[X_0:\cdots:X_{n+1}]
&\longmapsto
\bigl([X_0:\cdots:X_r], [X_{r+1}:\cdots:X_{n+1}]\bigr).
\end{align*}
Let $X \subset \mathbb{P}^{n+1}$ be a smooth hypersurface of degree $d\geq3$ for $n\geq 2$. 
Restricting this projection to $X$, we obtain a rational map
\[
f_{P^r,P^{n-r}} \colon X \dashrightarrow \mathbb{P}^r \times \mathbb{P}^{n-r}
\]
whose degree is given by
\[
\deg\bigl(f_{P^r,P^{n-r}}\bigr)=
\begin{cases}
d-2, & \text{if } P^r \subset X \text{ and } P^{n-r} \subset X,\\
d-1, & \text{if } P^r \subset X \text{ and } P^{n-r} \not\subset X,\\
d,   & \text{if } P^r \not\subset X \text{ and } P^{n-r} \not\subset X.
\end{cases}
\]
Note that when $r=0$, the map $f_{P^r,P^{n-r}}$ coincides with the projection from the point $P^0$.
By $[\ref{bio:th25g},\ \mathrm{Theorem}\ 3.5]$, if $(n,d) \neq (2,4)$ and the rational map
$f_{P^r,P^{n-r}}$ is Galois, then its Galois group is cyclic and is generated
by the automorphism given by the matrix
\[
\begin{pmatrix}
e_m I_{r+1} & 0 \\
0 & I_{n-r+1}
\end{pmatrix}
\]
where $e_m$ denotes a primitive $m$-th root of unity and
$m = \deg\bigl(f_{P^r,P^{n-r}}\bigr)$.
The following result is straightforward, and hence its proof is omitted.
\begin{thm}\label{back}
Let $\mac X \subset \mathbb{P}^{n+1}$ be a smooth hypersurface of degree $d$ with $n \geq 2$.
Let $P^r = W(X_{n-r+1}, \ldots, X_{n+1})$ and $P^{n-r} = W(X_{0}, \ldots, X_{n-r})$ be projective spaces, and let $f_{P^r,P^{n-r}} \colon \mac X \dashrightarrow \mathbb{P}^r \times \mathbb{P}^{n-r}$
be the rational map determined by $P^r$ and $P^{n-r}$.
Let $m$ denote the degree of this rational map.
If $\mac X$ admits an automorphism given by the matrix
$\begin{pmatrix}
e_m I_{r+1} & 0 \\
0 & I_{n-r+1}
\end{pmatrix}$,
then the rational map $f_{P^r,P^{n-r}}$ is Galois.
In particular, $\mathcal X / \langle g \rangle$ is a rational variety.
\end{thm}
Note that when $r=0$, the point $P^0$ is called a Galois point for $\mathcal X$ ([\ref{bio:my00},\ref{bio:y03}]).

We now briefly recall several results from previous studies on the rationality of quotient varieties arising from automorphisms of smooth hypersurfaces.
\begin{thm}\label{r1}$([\ref{bio:th21l},\ \mr{Theorem}\ 1.7,\  \mr{Theorem}\ 1.8,\ \mr{and}\ \mr{Theorem}\ 1.9])$.
Let $\mathcal X \subset \mathbb P^{n+1}$ be a smooth hypersurface of degree $d\ge 4$, 
and let $g\in \mathrm{PGL}(n+2,k)$ be an automorphism of $\mathcal X$.
\begin{enumerate}[$(i)$] 
\item We assume ${\rm ord}(g)=d-1$.
If one of the following holds, then $\mathcal X/\langle g\rangle$  is rational variety. 
\begin{enumerate}[$(a)$]
\item $n=1$ and $|{\rm Fix}(g)|\neq 2$;
\item $n=2$ and ${\rm Fix}(g)$ contains a non-rational curve;
\item $n\ge 3$ and ${\rm Fix}(g)$ has codimension one in $\mathcal X$.
\end{enumerate}
\item We assume ${\rm ord}(g)=d$.
If one of the following holds, then $\mac X/\langle g\rangle \cong \mathbb P^n$.
\begin{enumerate}[$(a)$]
\item $n=1$ and ${\rm Fix}(g)\neq \emptyset$;
\item $n\ge 2$ and ${\rm Fix}(g)$ has codimension one in $\mathcal X$.
\end{enumerate}
\end{enumerate}
\end{thm}
\begin{thm}\label{r2}$([\ref{bio:bb16n},\ {\rm Corollary}\ 24\ {\rm and}\ {\rm Corollary}\ 32])$.
Let $\mathcal X \subset \mathbb P^{2}$ be a smooth plane curve of degree $d\ge 4$, 
and let $g\in \mathrm{PGL}(3,k)$ be an automorphism of $\mathcal X$.
If ${\rm ord}(g)=k(d-1)$ or $kd$ for $k\ge 2$, then $\mac X/\langle g\rangle\cong \mathbb P^1$.
\end{thm}
In particular, in all cases covered by the above theorems, the hypersurface $\mathcal X$ admits a Galois point $P$, and the automorphism $g$ generates the Galois group of the projection from $P$.
Related phenomena for smooth curves on $\mathbb P^1\times \mathbb P^1$ are investigated in [\ref{bio:hs25l}].
In [\ref{bio:th23g},\ref{bio:th25a}],
groups consisting of automorphisms whose quotients are projective spaces are also studied.
The following theorem is also related to the notion of Galois skew lines in $\mathbb P^3$. 
Under these conditions, the quotient map is birationally equivalent to the projection from a pair of skew lines to $\mathbb P^1\times \mathbb P^1$, with a Galois group $\langle g\rangle$. See~[\ref{bio:th25g}] for details.
\begin{thm}\label{r4}$([\ref{bio:th25g},\ \mr{Theorem}\ 1.7])$.
Let $\mac X\subset \mathbb P^3$ be a smooth surface of degree $d\geq 5$, and let $g\in \mathrm{PGL}(4,k)$ be an automorphism of $\mathcal X$.
If one of the following holds, then $\mathcal X/\langle g\rangle$  is a rational surface.
\begin{enumerate}[$(a)$]
\item  ${\rm ord}(g)=d-2$ and ${\rm dim}\,{\rm Fix}(g)=1$;
\item ${\rm ord}(g)=d-1$ and ${\rm Fix}(g)$ contains a line;
\item  ${\rm ord}(g)=d$, ${\rm dim}\,{\rm Fix}(g)=0$, and $|{\rm Fix}(g)|\geq d+3$.
\end{enumerate}
\end{thm}

For variables $X_1,\ldots, X_m$ and $d\geq 0$, let 
$k[X_1,\ldots ,X_m]_d$ be the $k$-vector space of forms of degree $d$ in the variables $X_1,\ldots, X_m$.
Let $\mathbb P^{n+1}$ have homogeneous coordinates $[X_0:\cdots:X_{n+1}]$.
For each $i\in\{0,\ldots,n+1\}$, let
\[
P_i:=[0:\cdots:0:\underbrace{1}_{\text{$i$-th entry}}:0:\cdots:0]
\]
be the $i$-th coordinate point.
Let $\mathcal X\subset \mathbb P^{n+1}$ be a hypersurface defined by a homogeneous
polynomial
\[
F(X_0,\ldots,X_{n+1})=0
\]
of degree $d$.
If $P_i\notin \mathcal X$, then $F(X_0,\ldots,X_{n+1})$ contains the monomial $X_i^d$.
If $P_i\in \mathcal X$ and $\mathcal X$ is smooth at $P_i$, then
$F(X_0,\ldots,X_{n+1})$ contains a monomial of the form $X_i^{d-1}X_j$ for some $j\ne i$.

For $A=(a_{ij})\in{\rm GL}(n+2,k)$ and
$F(X_0,\ldots,X_{n+1})\in k[X_0,\ldots,X_{n+1}]_d$, 
we define the action of $A$ on $F(X_0,\ldots,X_{n+1})$ as follows:
\[ A^{\ast}F(X_0,\ldots,X_{n+1}):=F\left(\sum_{i=1}^{n+2}a_{1i}X_{i-1},\ldots,\sum_{i=1}^{n+2}a_{n+2\,i}X_{i-1}\right).\]
Let $\mac X\subset \mb P^{n+1}$ be a hypersurface of degree $d$ defined by $F(X_0,\ldots,X_{n+1})=0$.
 If there exists $t \in k^{\times}$ such that
\[A^{\ast}F(X_0,\ldots,X_{n+1})=tF(X_0,\ldots,X_{n+1}),\] then $A$ induces an automorphism of $\mac X$.

Let $g\in \mathrm{PGL}(n+2,k)$, and choose a representative
$A\in \mathrm{GL}(n+2,k)$.
We write $r(g)$ for the number of distinct eigenvalues of $A$.
This number is well-defined, i.e. independent of the choice of representative,
since replacing $A$ by a scalar multiple does not change its set of eigenvalues.
Let ${\rm ord}(g)=:m$ and set $r:=r(g)$. Then
\[
  r \leq \min\{m,\,n+2\}.
\]
Let \( \lambda_1,\ldots,\lambda_r \) be the distinct eigenvalues of \( A \), 
and let \( W_i \subset k^{n+2} \) be the corresponding eigenspace for $\lambda_i$.
We denote by \( \mathbb{P}(W_i)\subset \mathbb{P}^{n+1} \) the projective 
subspace associated with \( W_i \).
If $g$ is an automorphism of a hypersurface $\mathcal X\subset\mathbb P^{n+1}$, then 
\[
  \mathrm{Fix}(g) = \bigcup_{i=1}^r \mathbb{P}(W_i)\cap\mac X.
\]
If the codimension $\mathrm{codim}\,\mathrm{Fix}(g)$ of $\mathrm{Fix}(g)$ is at most $2$, then there exists some \( i \) such that
$\dim \bigl(\mathbb{P}(W_i)\cap \mathcal X\bigr) \geq n-2$,
and consequently $\dim \mathbb P(W_i)\ge n-2$.
Since $\dim \mathbb P(W_i)=\dim W_i-1$, this implies $\dim W_i\ge n-1$.
As $\sum_{i=1}^r \dim W_i=n+2$, it follows that $r\le 4$.
By multiplying \( A \) with a suitable element of \( k^{\times} \), 
we may assume that one of the eigenvalues of \( A \) is equal to \( 1 \).
When \(n=2\), i.e., in the case of \(\mathbb{P}^3\), we always have \(r \leq 4\).
The study of automorphisms of smooth projective hypersurfaces with fixed loci of codimension at most $2$ reduces to the study of automorphisms of smooth hypersurfaces in \(\mathbb{P}^3\).

For $a_1,\ldots,a_m\in k^{\times}$, we write
$D_{n+2}(a_1,\ldots,a_m,1_{n+2-m})$
for the diagonal matrix of size $(n+2)\times(n+2)$
whose diagonal entries are $a_1,\ldots,a_m$
followed by $n+2-m$ entries equal to $1$.
Let
\[
g=[D_{n+2}(a_1,\ldots,a_m,1_{n+2-m})]\in \mathrm{PGL}(n+2,k)
\]
be an automorphism of a smooth hypersurface $\mathcal X\subset \mathbb P^{n+1}$
such that ${\rm Fix}(g)$ has codimension at most two in $\mathcal X$,
and let $F$ be an irreducible component of ${\rm Fix}(g)$ of codimension
at most $2$ in $\mathcal X$.
Then necessarily $m\le 3$.
After permuting the diagonal entries, we may assume that $g$ and $F$ are given by one
of the following normal forms:
\begin{center}
\renewcommand{\arraystretch}{1.3}
\begin{tabular}{cll}
\textbf{Type I}   & $D_{n+2}(a_1,1_{n+1})$, 
                    & $F\subset\mathcal X\cap W(X_0)$;\\
\textbf{Type II}  & $D_{n+2}(a_1,a_1,1_{n})$, 
                    & $F\subset \mathcal X\cap W(X_0,X_1)$;\\
\textbf{Type III} & $D_{n+2}(a_1,a_2,1_{n})$, 
                    & $F\subset \mathcal X\cap W(X_0,X_1)$;\\
\textbf{Type IV}  & $D_{n+2}(a_1,a_1,a_1,1_{n-1})$, 
                    & $F=\mathcal X\cap W(X_0,X_1,X_2)$;\\
\textbf{Type V}   & $D_{n+2}(a_1,a_1,a_2,1_{n-1})$, 
                    & $F=\mathcal X\cap W(X_0,X_1,X_2)$;\\
\textbf{Type VI}  & $D_{n+2}(a_1,a_2,a_3,1_{n-1})$, 
                    & $F=\mathcal X\cap W(X_0,X_1,X_2)$.
\end{tabular}
\end{center}
Here $a_1,a_2,a_3\in k^{\times}$ are pairwise distinct.

In the above classification, the codimension of ${\rm Fix}(g)$ in $\mathcal X$ 
is described as follows.
For Type I, ${\rm Fix}(g)$ has codimension $1$ in $\mathcal X$.
For Types II and III, ${\rm Fix}(g)$ has codimension $1$ or $2$ in $\mathcal X$.
For Types IV--VI, ${\rm Fix}(g)$ has codimension $2$ in $\mathcal X$.
\section{Proof of Theorems \ref{1.1} and \ref{1.2}}
We prove Theorems~\ref{1.1} and~\ref{1.2} by considering separately
each of the diagonal normal forms listed in Section~3.
\subsection{Type I}
We first consider Type~I.
We write the defining equation of $\mathcal{X}$ as
\[
F(X_0,\ldots,X_{n+1})
= \sum_{i=0}^{d} F_{d-i}(X_1,\ldots,X_{n+1})X_0^{i}
\]
where $F_{d-i}(X_1,\ldots,X_{n+1})\in k[X_1,\ldots,X_{n+1}]_{d-i}$ for $i=0,\ldots,d$.
Since $g=[D(a,I_{n+1})]$ is an automorphism of $\mathcal{X}$, there exists $t\in k^{\times}$ such that
\[
D(a,I_{n+1})^*F(X_0,\ldots,X_{n+1})
= t\,F(X_0,\ldots,X_{n+1}).
\]
Since $[D(a,I_{n+1})]^*X_0=aX_0$ and $[D(a,I_{n+1})]^*X_i=X_i$ for $i=1,\ldots,n+1$, we obtain 
\begin{equation}\label{casei}
\sum_{i=0}^{d} a^iF_{d-i}(X_1,\ldots,X_{n+1})X_0^{i}=t
\sum_{i=0}^{d}F_{d-i}(X_1,\ldots,X_{n+1})X_0^{i}.
\end{equation}
\begin{pro}\label{casei,pro}
Let $\mathcal X \subset \mathbb P^{n+1}$ be a smooth hypersurface of degree $d$.
We assume that $\mathcal{X}$ admits an automorphism $g:=[D(a,I_{n+1})]$ where $a\in k^{\times}\setminus\{1\}$.
If $n\geq 3$, then ${\rm Fix}(g)$ does not contain a projective subspace of dimension $n-1$ or $n-2$.
\end{pro}
\begin{proof}
We assume that $n\geq3$.
It suffices to show that $\operatorname{Fix}(g)$ does not contain
a projective subspace of dimension $n-2$.
Since $\mathcal{X}$ is smooth, $F_d(X_1,\ldots,X_{n+1})\neq 0$.
Moreover, since $a\neq1$ and $F_d(X_1,\ldots,X_{n+1})\neq 0$, the equation~$(\ref{casei})$ implies that $F_{d-1}(X_1,\ldots,X_{n+1})=0$.
We assume that $W(X_0,X_i,X_j)\subset {\rm Fix}(g)$ for $1\leq i<j$.
Then
\[
F_d(X_1,\ldots,X_{n+1})=G(X_1,\ldots,X_{n+1})X_iX_j
\]
for some $G(X_1,\ldots,X_{n+1})\in k[X_1,\ldots,X_{n+1}]_{d-2}$.
Since $n\geq3$, the intersection
\[
\{G(X_1,\ldots,X_{n+1})=0\}\cap W(X_0,X_i,X_j).
\]
is nonempty.
Since $F_{d-1}(X_1,\ldots,X_{n+1})=0$,
every point in this intersection is a singular point of $\mathcal X$. 
This contradicts the smoothness of $\mathcal X$. 
Therefore, if $n\geq 3$, then ${\rm Fix}(g)$ does not contain a linear subspace of dimension $n-1$ or $n-2$.
\end{proof}
\begin{thm}\label{casei,thm}
Let $\mathcal{X}\subset \mathbb{P}^{n+1}$ be a smooth hypersurface of degree $d$ for $n\geq2$.
We assume that $\mathcal{X}$ admits an automorphism $g:=[D(a,I_{n+1})]$ where $a\in k^{\times}\setminus\{1\}$.
Then the order $\mr{ord}(g)$ of $g$ divides 
\[
d\quad \mr{or}\quad d-1.
\]
Moreover, if $\operatorname{ord}(g) = d$ or $d-1$, then
$\mathcal{X} / \langle g \rangle$ is a rational variety.
\end{thm}
\begin{proof}
Since $\mathcal{X}$ is smooth, we have $F_d(X_1,\ldots,X_{n+1})\neq 0$.
By the equation~$(\ref{casei})$,
\[F_d(X_1,\ldots, X_{n+1})=tF_d(X_1,\ldots, X_{n+1}).\]
It follows that 
\[
t=1.
\]

We assume that $P_0\notin \mathcal{X}$. Then $F_0(X_1,\ldots,X_{n+1})\neq 0$. 
Comparing the coefficients of $X_0^{d}$ in \eqref{casei}, we obtain
\[
a^d=1.
\]
Therefore, $\mr{ord}(g)$ divides $d$.

We assume that $P_0\in \mathcal{X}$. Then $F_0(X_1,\ldots,X_{n+1})=0$.
Since $\mathcal X$ is smooth at $P_0$, $F_1(X_1,\ldots,X_{n+1})\neq0$. 
Comparing the coefficients of $X_0^{d-1}$ in \eqref{casei}, we obtain
\[
a^{d-1}=1.
\]
Hence, $\mr{ord}(g)$ divides $d-1$.

We assume that $\operatorname{ord}(g)=d$ or $d-1$.  
By Proposition~\ref{casei,pro} and Theorem~\ref{r1}, $\mathcal{X}/\langle g\rangle$ is a rational variety.
\end{proof}
By Theorem~\ref{casei,thm}, under Type I, the order $\operatorname{ord}(g)$ divides $d$ or $d-1$. 
Consequently, if $\operatorname{ord}(g)\ge 3$, then $\operatorname{ord}(g)$ cannot divide $d-2$.
In particular, the case $\operatorname{ord}(g)=d-2$ does not occur.

\subsection{Types II and III}
We next consider Types~II and~III, which can be treated simultaneously,
since both correspond to diagonal actions affecting only the coordinates $X_0$ and $X_1$.
We write the defining equation of $\mathcal{X}$ as
\[
F(X_0,\ldots,X_{n+1})
= \sum_{i,j=0}^{d} F_{i,j}(X_2,\ldots,X_{n+1})X_0^{i}X_1^j
\]
where $F_{i,j}(X_2,\ldots,X_{n+1})\in k[X_2,\ldots,X_{n+1}]_{d-i-j}$ for $i,j=0,\ldots,d$.
 We assume that $\mathcal{X}$ has an automorphism $[D(a,b,I_{n})]$ where $a,b\in k^{\times}\backslash\{1\}$.
There exists $t \in k^{\times}$ such that $D(a,b,I_n)^{*} F(X_0,\ldots,X_{n+1})= tF(X_0,\ldots,X_{n+1})$, and hence
\begin{equation}\label{casesii,iii}
\sum_{i,j=0}^{d} a^ib^jF_{i,j}(X_2,\ldots,X_{n+1})X_0^{i}X_1^j
= t\sum_{i,j=0}^{d} F_{i,j}(X_2,\ldots,X_{n+1})X_0^{i}X_1^j.
\end{equation}
We define 
\begin{equation*}
\begin{split}
F_{u\,_{\underline{0,1}}}(X_0,\ldots,X_{n+1}):=&\sum_{i+j=u}F_{i,j}(X_2,\ldots,X_{n+1})\, X_0^iX_1^j
\end{split}
\end{equation*}
for $u=0,\ldots, d$.
\begin{lem}\label{casesii,iii,lem1}
We have that $\sum_{u=0,1} F_{u\,_{\underline{0,1}}}(X_0,\ldots,X_{n+1})\neq0$.
\end{lem}
\begin{proof}
If $\sum_{u=0,1} F_{u\,_{\underline{0,1}}}(X_0,\ldots,X_{n+1})=0$, then $W(X_0,X_1)\subset \mathcal X$.
Then every point of $W(X_0,X_1)$ is a singular point of $\mathcal{X}$.
This contradicts the smoothness of $\mathcal{X}$.
\end{proof}
After possibly permuting the coordinates $X_0$ and $X_1$, it suffices to consider the following cases:
\begin{itemize}
\item \(\mathcal{X} \cap \{P_0, P_1\} = \emptyset\);
\item \(\mathcal{X} \cap \{P_0, P_1\} = \{P_1\}\);
\item \(\mathcal{X} \cap \{P_0, P_1\} = \{P_0, P_1\}\).
\end{itemize}
In what follows, we work under this setup and focus on Types~II and~III.
\subsection{Type II}
We now consider Type II.
Throughout this subsection, we assume that $a=b$.
\begin{lem}\label{caseii,lem}
Let $\mathcal{X}\subset \mathbb{P}^{n+1}$ be a smooth hypersurface of degree $d$ for $n\geq2$.
We assume that $\mathcal{X}$ admits an automorphism $g:=[D(a,a,I_{n})]$ where $a\in k^{\times}\setminus\{1\}$.
Then one of the following holds:
\begin{enumerate}[$(i)$]
\item \( t = a \), $n=2$, and 
\[\sum_{u=0,1} F_{u\,_{\underline{0,1}}}(X_0,\ldots,X_{n+1})= F_{1,0}(X_2,\ldots,X_{n+1})X_0 +F_{0,1}(X_2,\ldots,X_{n+1})X_1.\]
\item \( t = 1 \) and \[\sum_{u=0,1} F_{u\,_{\underline{0,1}}}(X_0,\ldots,X_{n+1})= F_{0,0}(X_2,\ldots,X_{n+1}).\]
\end{enumerate}
\end{lem}
\begin{proof}
By the equation (\ref{casesii,iii}),
\[
\sum_{0\leq i+j\leq 1} a^{i+j}F_{i,j}(X_2,\ldots,X_{n+1})X_0^{i}X_1^j
= t\sum_{0\leq i+j\leq 1}F_{i,j}(X_2,\ldots,X_{n+1})X_0^{i}X_1^j.
\]
By Lemma \ref{casesii,iii,lem1}, $\sum_{u=0,1} F_{u\,_{\underline{0,1}}}(X_0,\ldots,X_{n+1})\neq0$.
By the equation (\ref{casesii,iii}), we have that $t=a$ or $t=1$.
Moreover,
if $t=a$ (resp. $t=1$), then $F_{0,0}(X_2,\ldots,X_{n+1})=0$ (resp. $F_{1,0}(X_2,\ldots,X_{n+1})X_0 +F_{0,1}(X_2,\ldots,X_{n+1})X_1=0$).

We assume that $t=a$.
Then
\[
\begin{split}
F(X_0,\ldots ,X_{n+1})=&\sum_{2\leq i+j\leq d}F_{i,j}(X_2,\ldots,X_{n+1})X^i_0X^j_1\\
&\ \ +F_{1,0}(X_2,\ldots,X_{n+1})X_0 +F_{0,1}(X_2,\ldots,X_{n+1})X_1.
\end{split}
\]
If $n\geq 3$, then the intersection
\[
\{F_{1,0}(X_2,\ldots,X_{n+1})=0\}\cap \{F_{0,1}(X_2,\ldots,X_{n+1})=0\}\cap W(X_0,X_1)
\]
is nonempty.
Every point in this intersection is a singular point of $\mathcal X$. 
This contradicts the smoothness of $\mathcal X$. 
Thus, $n=2$.
\end{proof}
\begin{pro}\label{caseii,pro}
Let $\mathcal X \subset \mathbb P^{n+1}$ be a smooth hypersurface of degree $d$ for $n\geq2$.
We assume that $\mathcal{X}$ admits an automorphism $g:=[D(a,a,I_{n})]$ where $a\in k^{\times}\backslash\{1\}$.
If $n\geq 3$, then ${\rm Fix}(g)$ does not contain a projective subspace of dimension $n-2$ or $n-1$.
\end{pro}
\begin{proof}
We assume that $n\geq3$.
It suffices to show that $\operatorname{Fix}(g)$ does not contain a projective subspace of dimension $n-2$.
We assume that $W(X_0,X_1,X_i)\subset{\rm Fix}(g)$ for $2\leq i$.
By Lemma \ref{caseii,lem}, $\sum_{u=0,1} F_{u\,_{\underline{0,1}}}(X_0,\ldots,X_{n+1})= F_{0,0}(X_2,\ldots,X_{n+1})$.
Since $W(X_0,X_1,X_i)\subset {\rm Fix}(g)$, 
\[
F_{0,0}(X_2,\ldots,X_{n+1})=G(X_2,\ldots,X_{n+1})X_i
\]
for some $G(X_2,\ldots,X_{n+1})\in k[X_2,\ldots,X_{n+1}]_{d-1}$.
Since $n\geq3$, the intersection
\[
\{G(X_2,\ldots,X_{n+1})=0\}\cap W(X_0,X_1)
\]
is nonempty.
Since $F_{1\,_{\underline{0,1}}}(X_0,\ldots,X_{n+1})=0$,
every point in this intersection is a singular point of $\mathcal X$. 
This contradicts the assumption that $\mathcal X$ is smooth. 
Therefore, if $n\geq3$, then ${\rm Fix}(g)$ does not contain a linear subspace of dimension $n-1$ or $n-2$.
\end{proof}
\begin{thm}\label{caseii,thm}
Let $\mathcal{X}\subset \mathbb{P}^{n+1}$ be a smooth hypersurface of degree $d$ for $n\geq2$.
We assume that $\mathcal{X}$ admits an automorphism $g:=[D(a,a,I_{n})]$ where $a\in k^{\times}\setminus\{1\}$.
Then the order $\mr{ord}(g)$ of $g$ divides one of the integers
\[
d,\quad d-1,\quad \text{or}\quad d-2.
\]
Moreover, the case where $\mr{ord}(g)\geq 3$ and $\mr{ord}(g)$ divides $d-2$ occurs only when $n=2$.
If $\mr{ord}(g)\geq 3$ and $\operatorname{ord}(g)\in\{d,d-1,d-2\}$, then $\mathcal{X} / \langle g \rangle$ is a rational variety.
\end{thm}
\begin{proof}
We first assume that \(\mathcal{X} \cap \{P_0, P_1\} = \emptyset\).
Then $F(X_0,\ldots, X_{n+1})$ contains the monomials $X_0^{d}$ and $X_1^{d}$. 
By the equation $\eqref{casesii,iii}$, we have 
\[
t=a^{d}.
\]
By Lemma~\ref{caseii,lem},
\begin{equation*}
a^{d}=
\left\{
\begin{aligned}
a\ \ \mr{and}\ \ n=2&\quad(\mr{case}~(i)\ \mr{of\ Lemma}~\ref{caseii,lem}),\\
1&\quad(\mr{case}~(ii)\ \mr{of\ Lemma}~\ref{caseii,lem}).
\end{aligned}
\right.
\end{equation*}
Consequently, $a^d=a$ and $n=2$ or $a^d=1$.
Thus, 
$\mr{ord}(g)$ divides $d-1$ and $n=2$, or $\mr{ord}(g)$ divides $d$.
\\

Next, we assume that $\mathcal{X} \cap \{P_0, P_1\} = \{P_1\}$.  
Then $F(X_0,\ldots,X_{n+1})$ contains the monomials $X_0^{d}$ and 
$X_i X_1^{d-1}$ for some $i\neq 1$.
By equation~\eqref{casesii,iii}, we obtain
\[
t=a^{d}.
\]
If $i\neq 0$, then the equation~\eqref{casesii,iii} implies $t=a^{d-1}$.
By $t=a^d$, we have $a=1$. This contradicts that $a\neq 1$. 
Thus $i=0$.
Since $t=a^d$, the argument reduces to the case where $\mathcal X \cap \{P_0, P_1\} = \emptyset$.
Hence, $\mr{ord}(g)$ divides $d-1$ and $n=2$, or $\mr{ord}(g)$ divides $d$.
\\

Finally, we assume that $\mathcal{X} \cap \{P_0, P_1\} = \{P_0,P_1\}$. 
Then $F(X_0,\ldots,X_{n+1})$ contains
the monomials $X_i X_0^{d-1}$ and $X_j X_1^{d-1}$ for $i\neq 0$ and $j\neq 1$.
If $i=1$ and $j\neq 0$, then the equation~\eqref{casesii,iii} gives 
\[
a^{d}=a^{d-1},
\]
which contradicts the assumption $a\neq 1$. 
Thus, the case $i=1$ and $j\neq 0$ cannot occur.  
Similarly, the case $i\neq 1$ and $j=0$ is impossible.  
Hence, there are two possibilities:
\[
(i,j)=(1,0)\qquad\text{or}\qquad i\neq 1\ \text{and}\ j\neq 0.
\]

We assume that $(i,j)=(1,0)$. By the equation~\eqref{casesii,iii}, 
\[
t=a^{d}.
\]
The remaining argument is reduced to the case where $\mathcal X \cap \{P_0, P_1\} = \emptyset$.
Hence, $\mr{ord}(g)$ divides $d-1$ and $n=2$, or $\mr{ord}(g)$ divides $d$.

We assume that $i\neq 1$ and $j\neq 0$.
By the equation~\eqref{casesii,iii}, we have
\[
t=a^{d-1}.
\]
By Lemma~\ref{caseii,lem},
\begin{equation*}
a^{d-1}=
\left\{
\begin{aligned}
a\ \ \mr{and}\ \ n=2&\quad(\mr{case}~(i)\ \mr{of\ Lemma}~\ref{caseii,lem}),\\
1&\quad(\mr{case}~(ii)\ \mr{of\ Lemma}~\ref{caseii,lem}).
\end{aligned}
\right.
\end{equation*}
Hence, $\mr{ord}(g)$ divides $d-2$ and $n=2$ or $\mr{ord}(g)$ divides $d-1$. 

In particular, since $\gcd(d-2,d)=1$ or $2$, it follows that if $\operatorname{ord}(g)\ge 3$ and $\operatorname{ord}(g)$ divides $d-2$, then necessarily $n=2$.

We assume that $\operatorname{ord}(g)=d$ or $d-1$. By the assumption on $\operatorname{Fix}(g)$ and Proposition~\ref{caseii,pro}, Theorem~\ref{r1} implies that $\mathcal{X}/\langle g\rangle$ is a rational variety.
If $\operatorname{ord}(g)=d-2$ and $\mr{ord}(g)\geq 3$, then by the assumption on $\operatorname{Fix}(g)$ and Theorem~\ref{r4}, $\mathcal{X}/\langle g\rangle$ is a rational surface.
\end{proof}

\subsection{Type III}
We now consider Type III and assume that $a\neq b$.
\begin{lem}\label{caseiii,1}
Let $\mathcal{X}\subset \mathbb{P}^{n+1}$ be a smooth hypersurface of degree $d$ where $n\geq2$.
We assume that $\mathcal{X}$ admits an automorphism $g:=[D(a,b,I_{n})]$ where $a,b\in k^{\times}\backslash\{1\}$ with $a\ne b$. 
Then $t = 1$ and 
\[\sum_{u=0,1} F_{u\,_{\underline{0,1}}}(X_0,\ldots,X_{n+1})= F_{0,0}(X_2,\ldots,X_{n+1}).\]
\end{lem}
\begin{proof}
The argument is similar to that in the proof of Lemma~\ref{caseii,lem}.
We omit the details.
\end{proof}
By Lemma \ref{caseiii,1}, we see that $\mr{Fix}(g)$ has codimesion two in $\mac X$.  
\begin{pro}\label{caseiii,pro}
Let $\mathcal X \subset \mathbb P^{n+1}$ be a smooth hypersurface of degree $d$ for $n\geq3$.
We assume that $\mathcal{X}$ admits an automorphism $g:=[D(a,b,I_{n})]$ where $a,b\in k^{\times}\backslash\{1\}$ with $a\neq b$.
Then ${\rm Fix}(g)$ does not contain a projective subspace of dimension $n-2$ or $n-1$.
\end{pro}
\begin{proof}
The argument follows the same lines as the proof of Proposition~\ref{caseii,pro}. We omit the details.
\end{proof}
\begin{pro}\label{caseiii,pro2}
Let $\mathcal X \subset \mathbb P^{3}$ be a smooth hypersurface of degree $d$.
We assume that $\mathcal{X}$ admits an automorphism $g:=[D(a,b,I_2)]$ where $a,b\in k^{\times}\backslash\{1\}$ with $a\neq b$.
Then $|{\rm Fix}(g)|\geq d$.
\end{pro}
\begin{proof}
We write $F_{0,0}(X_2,X_3)= u \prod_{i=1}^d (a_i X_2 - b_i X_3)$ where $u,a_i,b_i \in k$.
We assume that $[0:0:b_i:a_i]=[0:0:b_j:a_j]$ for some $i\neq j$.
By Lemma \ref{caseiii,1}, $\sum_{u=0,1} F_{u_{\underline{0,1}}}(X_0,X_1,X_2,X_3)= F_{0,0}(X_2,X_3)$.
Hence the point $[0:0:b_i:a_i]$ is a singular point of $\mathcal{X}$.
This contradicts the smoothness of $\mathcal{X}$.
Therefore, we conclude that
$[0:0:b_i:a_i]\neq[0:0:b_j:a_j]$ for $i\neq j$, and consequently,
$|\mathrm{Fix}(g)| \geq d$.
\end{proof}
\begin{thm}\label{caseiii,thm}
Let $\mathcal{X}\subset \mathbb{P}^{n+1}$ be a smooth hypersurface of degree $d$ for $n\geq2$.
We assume that $\mathcal{X}$ admits an automorphism $g:=[D(a,b,I_{n})]$ where $a,b\in k^{\times}\setminus\{1\}$ with $a\neq b$.
Then the order of $g$ divides one of the integers
\[
(d-1)d,\quad(d-1)^2,\quad \text{or}\quad (d-2)d.
\]
If $\operatorname{ord}(g) = kd$ or $k(d-1)$ for $k\geq2$, then $\mathcal{X} / \langle g \rangle$ is a rational variety.
\end{thm}
\begin{proof}
We first assume that \(\mathcal{X} \cap \{P_0, P_1\} = \emptyset\).
Then $F(X_0,\ldots, X_{n+1})$ contains the monomials $X_0^{d}$ and $X_1^{d}$. 
By the equation $\eqref{casesii,iii}$ and Lemma~\ref{caseiii,1}, we have 
\[
t=a^{d}=b^{d}=1.
\]
Then $\mr{ord}(g)$ divides $d$.
\\

Next, we assume that \(\mathcal{X} \cap \{P_0, P_1\} = \{P_1\}\).
Then $F(X_0,\ldots, X_{n+1})$ contains the monomials $X_0^{d}$, and $X_iX_1^{d-1}$ for $i\neq1$. 

We assume that $i=0$.
By the equation $\eqref{casesii,iii}$ and Lemma~\ref{caseiii,1}, 
\[
t=a^{d}=ab^{d-1}=1.
\]
By substituting \(a=b^{1-d}\) into \(a^d=1\), we get $b^{(d-1)d}=1$.
Since
\[
a=b^{1-d}\quad \mr{and}\quad b^{(d-1)d}=1,
\]
$\mr{ord}(g)$ divides $(d-1)d$.
In addition, if $\operatorname{ord}(g) =k(d-1)$ for $k\geq2$, then $g^k=[D(1,b^k,I_n)]$.
Theorem~\ref{r1} implies that $\mathcal{X} / \langle g^k \rangle$ is rational, and hence $X/\langle g\rangle$ is rational.
We assume that $\operatorname{ord}(g) = kd$ for $k\geq2$.
Then $g^k=[D(b^k,b^k,I_n)]$.
Since $P_0\not\in\mac X$, $W(X_2,\ldots,X_{n+1})\not\subset \mac X$.
By Proposition~\ref{caseiii,pro}, $W(X_0,X_1)\not\subset \mac X$.
By Theorem \ref{back}, $\mathcal{X} / \langle g^k \rangle$ is rational, and hence $X/\langle g\rangle$ is rational.

We assume that $i\neq0$.
By the equation $\eqref{casesii,iii}$ and Lemma~\ref{caseiii,1}, 
\[
t=a^{d}=b^{d-1}=1.
\]
Then $\mr{ord}(g)$ divides $(d-1)d$.
In addition, if $\operatorname{ord}(g) =k(d-1)$ (resp. $kd$) for $k\geq2$, then  $g^k=[D(1,b^k,I_n)]$ (resp. $g^k=[D(a^k,I_{n+1})]$).
Theorem~\ref{r1} implies that $\mathcal{X} / \langle g^k \rangle$ is rational, and hence $X/\langle g\rangle$ is rational.
\\

Finally, we assume that \(\mathcal{X} \cap \{P_0, P_1\} = \{P_0,P_1\}\).
Then $F(X_0,\ldots, X_{n+1})$ contains the monomials $X_iX_0^{d-1}$ and $X_jX_1^{d-1}$ for $i\neq0$ and $j\neq1$. 

We assume that $(i,j)=(1,0)$.
By the equation $\eqref{casesii,iii}$ and Lemma~\ref{caseiii,1},
\[
t=ba^{d-1}=ab^{d-1}=1.
\]
By substituting \(a=b^{1-d}\) into \(ba^{d-1}=1\), we get $b^{(d-2)d}=1$.
Since
\[
a=b^{1-d}\quad \mr{and}\quad b^{(2-d)d}=1,
\]
$\mr{ord}(g)$ divides $(d-2)d$.
In addition, if $\operatorname{ord}(g) = kd$ for $k\geq2$.
Then $g^k=[D(b^k,b^k,I_n)]$.
Since $P_0\not\in\mac X$, we see that $W(X_2,\ldots,X_{n+1})\not\subset \mac X$.
By Proposition~\ref{caseiii,pro}, $W(X_0,X_1)\not\subset \mac X$.
By Theorem \ref{back}, $\mathcal{X} / \langle g^k \rangle$ is rational, and hence $\mathcal{X} / \langle g \rangle$ is rational.

We assume that $i=1$ and $j\neq0$.
By the equation $\eqref{casesii,iii}$ and Lemma~\ref{caseiii,1},
\[
t=ba^{d-1}=b^{d-1}=1.
\]
By substituting \(b=a^{1-d}\) into \(b^{d-1}=1\), we get $a^{(d-1)^2}=1$.
Since
\[
a^{(d-1)^2}=1\quad \mr{and}\quad b=a^{1-d},
\]
$\mr{ord}(g)$ divides $(d-1)^2$.
In addition, if $\operatorname{ord}(g) =k(d-1)$ for $k\geq2$, then  $g^k=[D(a^k,I_{n+1})]$.
Theorem~\ref{r1} implies that $\mathcal{X} / \langle g^k \rangle$ is rational, and hence $X/\langle g\rangle$ is rational.

If $i \neq 1$ and $j = 0$, then by interchanging $X_0$ and $X_1$ the situation is reduced to the case $i = 1$ and $j \neq 0$. In particular, $\mr{ord}(g)$ divides $(d-1)^2$.
In particular, if $\operatorname{ord}(g) =k(d-1)$ for $k\geq2$, then $X/\langle g\rangle$ is rational.

We assume that $i\neq1$ and $j\neq 0$.
By the equation $\eqref{casesii,iii}$ and Lemma~\ref{caseiii,1},
\[
t=a^{d-1}=b^{d-1}=1.
\]
Then $\mr{ord}(g)$ divides $d-1$.
\end{proof}

\subsection{Types IV, V, and VI}
Next, we turn to Types IV, V, and VI.
In these cases, the fixed locus ${\rm Fix}(g)$ has codimension $2$ in $\mathcal X$.
We write the defining equation of $\mathcal{X}$ as
\[
F(X_0,\ldots,X_{n+1})
= \sum_{i,j,k=0}^{d} F_{i,j,k}(X_3,\ldots,X_{n+1})X_0^{i}X_1^jX_2^k
\]
where $F_{i,j,k}(X_3,\ldots,X_{n+1})\in k[X_3,\ldots,X_{n+1}]_{d-i-j-k}$ for $i,j,k=0,\ldots,d$.
We assume that $\mathcal{X}$ admits an automorphism $[D(a,b,c,I_{n-1})]$ where $a,b,c\in k^{\times}\backslash\{1\}$. 
Then $D(a,b,c,I_{n-1})^{*} F(X_0,\ldots,X_{n+1})
= tF(X_0,\ldots,X_{n+1})$ where some $t \in k^{\times}$, and hence
\begin{equation}\label{casesiv,v,vi,u}
\sum_{i,j,k=0}^{d} a^ib^jc^kF_{i,j,k}(X_3,\ldots,X_{n+1})X_0^iX_1^jX_2^k
= t\sum_{i,j,k=0}^{d} F_{i,j,k}(X_3,\ldots,X_{n+1})X_0^iX_1^jX_2^k.
\end{equation}
For $u=0,\ldots, d$,
we define 
\begin{equation*}
\begin{split}
F_{u\,_{\underline{0,1,2}}}(X_0,\ldots,X_{n+1}):=&\sum_{i+j+k=u}F_{i,j,k}(X_3,\ldots,X_{n+1})X_0^iX_1^jX_2^k.
\end{split}
\end{equation*}
\begin{lem}\label{casesiv,v,vi,lem1}
We have that $\sum_{u=0,1} F_{u\,_{\underline{0,1,2}}}(X_0,\ldots,X_{n+1})\neq0$.
\end{lem}
\begin{proof}
The argument is similar to that in the proof of Lemma~\ref{casesii,iii,lem1}.
We omit the details.
\end{proof}
\subsection{Type IV}
We now treat Type IV and assume that $a=b=c$.
\begin{lem}\label{caseiv,lem1}
Let $\mathcal{X}\subset \mathbb{P}^{n+1}$ be a smooth hypersurface of degree $d$ for $n\geq2$.
We assume that $W(X_0,X_1,X_2)\subset \mathcal X$ and $\mathcal{X}$ admits an automorphism $[D(a,a,a,I_{n-1})]$ where $a\in k^{\times}\backslash\{1\}$.
Then $n\leq 4$, $t = a$, and
\[
\begin{split}
\sum_{u=0,1} F_{u\,_{\underline{0,1,2}}}(X_0,\ldots,X_{n+1})=&F_{1,0,0}(X_3,\ldots,X_{n+1})X_0+F_{0,1,0}(X_3,\ldots,X_{n+1})X_1\\
&+F_{0,0,1}(X_3,\ldots,X_{n+1})X_2.
\end{split}
\]
\end{lem}
\begin{proof}
As in Lemma~\ref{caseii,lem}, we see that $t = a$ or $t = 1$.  
We suppose that $t=1$. Then $\sum_{u=0,1} F_{u_{\underline{0,1,2}}}(X_0,\ldots,X\,_{n+1})=F_{0,0,0}(X_3,\ldots,X_{n+1})$.
Since $F_{0,0,0}(X_3,\ldots,X_{n+1})$ $\neq0$,
this contradicts that $W(X_0, X_1, X_2) \subset \mathcal{X}$.
Thus, $t=a$.
If $n\geq5$, then the intersection
\[
\begin{split}
\{F_{1,0,0}(X_3,\ldots,X_{n+1})=0\}\cap \{&F_{0,1,0}(X_3,\ldots,X_{n+1})=0\}\\
&\cap \{F_{0,0,1}(X_3,\ldots,X_{n+1})=0\}\cap W(X_0,X_1,X_2)
\end{split}
\]
is nonempty.
Then every point in this intersection is a singular point of $\mathcal X$.
This contradicts the smoothness of $\mathcal X$. 
Thus, $n\leq 4$.
\end{proof}
\begin{lem}\label{caseiv,lem2}
Let $\mathcal{X}\subset \mathbb{P}^{n+1}$ be a smooth hypersurface of degree $d$ for $n\geq2$.
We assume that $\mathcal{X}$ admits an automorphism $g:=[D(a,a,a,I_{n-1})]$ where $a\in k^{\times}\backslash\{1\}$. 
Then 
\[
\sum_{u=d-1,d} F_{u\,_{\underline{0,1,2}}}(X_0,\ldots,X_{n+1})\neq0,
\]
and  one of the following holds:
\begin{enumerate}[$(i)$]
\item \( t = a^d \), $W(X_3,\ldots, X_{n+1})\not\subset \mac X$, and 
\[
F_{d-1\,_{\underline{0,1,2}}}(X_0,\ldots,X_{n+1})=0.
\]
\item \( t =a^{d-1} \), $n=4$, $W(X_3,\ldots, X_5)\subset \mac X$, and
\[
F_{d\,_{\underline{0,1,2}}}(X_0,\ldots,X_5)=0.
\]
\end{enumerate}
\end{lem}
\begin{proof}
If $\sum_{u=d-1,d} F_{u\,_{\underline{0,1,2}}}(X_0,\ldots,X_{n+1})=0$, then the points $P_0$, $P_1$, and $P_2$ are singular
points of $\mathcal X$.
This contradicts the smoothness of $\mathcal X$.
Therefore, we have that $\sum_{u=d-1,d} F_{u\,_{\underline{0,1,2}}}(X_0,\ldots,X_{n+1})\neq0$.
By the equation $(\ref{casesiv,v,vi,u})$,
we have that if $F_{d\,_{\underline{0,1,2}}}(X_0,\ldots,X_{n+1})\neq0$
then $t=a^d$ whereas if $F_{d-1\,_{\underline{0,1,2}}}(X_0,\ldots,X_{n+1})\neq0$ then $t=a^{d-1}$.
Since $a\neq 1$, $a^d\neq a^{d-1}$.  
Thus, exactly one of the above cases occurs.

We assume that $F_{d\,_{\underline{0,1,2}}}(X_0,\ldots,X_{n+1})\neq0$.
If necessary, by permuting the variables $X_0, X_1, X_2$, we may assume that
$F_{d,0,0}(X_3,\ldots,X_{n+1}) \neq 0$.
Then $P_0\not\in \mathcal{X}$.
Consequently, $W(X_3,\ldots,X_{n+1})\not\subset\mathcal{X}$.

We assume that $F_{d-1\,_{\underline{0,1,2}}}(X_0,\ldots,X_{n+1})\neq0$.
By Lemma~$\ref{caseiv,lem1}$,
\[
\begin{split}
F(X_0,\ldots,X_{n+1})
&= \sum_{i,j,k=1}^{d-1} F_{i,j,k}(X_3,\ldots,X_{n+1})X_0^{i}X_1^jX_2^k.
\end{split}
\]
Then $W(X_3,\ldots,X_{n+1})\subset\mathcal{X}$.
If $n=2$, then $F_{i,j,k}(X_3)\in k[X_3]_{d-(i+j+k)}$.
Since \(1 \le i+j+k \le d-1\), it follows that
\(F(X_0,X_1,X_2,X_3)\) is divisible by \(X_3\).
This contradicts the smoothness of \(\mathcal X\).
We assume that \(n=3\).
Since
$g=[D(a,a,a,I_2)]=[D(I_3,a,a)]$,
we are reduced to Case~(ii).
Moreover, as \(i+j+k \le d-1\), this situation falls under
case~(i) of Lemma~\ref{caseii,lem}.
Consequently, \(n\) must be equal to \(2\),
which contradicts the assumption that \(n=3\).
\end{proof}
\begin{thm}\label{caseiv,thm}
Let $\mathcal{X}\subset \mathbb{P}^{n+1}$ be a smooth hypersurface of degree $d$ for $n\geq2$.
We assume that $\mathcal{X}$ admits an automorphism $g:=[D(a,a,a,I_{n-1})]$ and $W(X_0,X_1,X_2)\subset \mathcal X$ where $a\in k^{\times}\setminus\{1\}$.
Then $n\leq4$, and the order $\mr{ord}(g)$ of $g$ divides 
\[
d-1\quad \text{or}\quad d-2.
\]
Moreover, the case where $\mr{ord}(g)$ divides $d-2$ occurs only when $n=4$. 
If $\operatorname{ord}(g) = d-1$ or $d-2$, then
$\mathcal{X} / \langle g \rangle$ is a rational variety.
\end{thm}
\begin{proof}
By Lemma $\ref{caseiv,lem1}$, we have
\[
t=a.
\]
By Lemma $\ref{caseiv,lem2}$,
\[
a=
\left\{
\begin{aligned}
&a^d&\ \ (\mr{case}~(i)\ \mr{of\ Lemma}~\ref{caseiv,lem2}),\\
&a^{d-1}\ \ \mr{and}\ \ n=4&\ \ (\mr{case}~(ii)\ \mr{of\ Lemma}~\ref{caseiv,lem2}).
\end{aligned}
\right.
\]
Thus, $\operatorname{ord}(g)$ divides $d-1$ or $d-2$.
In particular, if $\operatorname{ord}(g)=d-1$ (resp.\ $d-2$), then we are in case~(i) of Lemma~\ref{caseiv,lem2} (resp.\ case~(ii) of Lemma~\ref{caseiv,lem2}). In this situation, we have
$W(X_3,\ldots,X_{n+1}) \not\subset \mathcal{X}$ (resp. $W(X_3,\ldots,X_5) \subset \mathcal{X}$).
Therefore, by Theorem~\ref{back}, $\mathcal{X}/\langle g\rangle$ is rational.
\end{proof}

\subsection{Type V}
We now consider Type V and assume that $a=b$ and $a\neq c$.
\begin{lem}\label{casev,1}
Let $\mathcal{X}\subset \mathbb{P}^{n+1}$ be a smooth hypersurface of degree $d$ for $n\geq2$.
We assume that $W(X_0,X_1,X_2)\subset \mathcal X$ and
$\mathcal{X}$ admits an automorphism $g:=[D(a,a,c,I_{n-1})]$
where $a,c\in k^{\times}\backslash\{1\}$ with $a\neq c$.
Then one of the following holds:
\begin{enumerate}[$(i)$]
\item \( t = a \), $n\leq 3$, and
\begin{equation*}
\begin{split}
\sum_{u=0,1} F_{u\,_{\underline{0,1,2}}}(X_0,\ldots,X_{n+1})
=F_{1,0,0}(X_3,\ldots,X_{n+1})X_0+ F_{0,1,0}(X_3,\ldots,X_{n+1})X_1.
\end{split}
\end{equation*}
\item \( t = c \), $n=2$, and
\[
\sum_{u=0,1} F_{u\,_{\underline{0,1,2}}}(X_0,X_1,X_2,X_{3})= F_{0,0,1}(X_3)X_2 .
\]
\end{enumerate}
\end{lem}
\begin{proof}
The argument is similar to that in the proof of Lemma~\ref{caseii,lem}.
We omit the proof.
%
%
\end{proof}

After permuting $X_0$ and $X_1$ if necessary, 
it suffices to consider the following possible configurations of 
\(\mathcal{X} \cap \{P_0, P_1,P_2\}\):
\begin{itemize}
\item \(\mathcal{X} \cap \{P_0, P_1,P_2\} = \emptyset\);
\item \(\mathcal{X} \cap \{P_0, P_1,P_2\} = \{P_2\}\);
\item \(\mathcal{X} \cap \{P_0, P_1,P_2\} = \{P_1\}\);
\item \(\mathcal{X} \cap \{P_0, P_1,P_2\} = \{P_1,P_2\}\);
\item \(\mathcal{X} \cap \{P_0, P_1,P_2\} = \{P_0,P_1\}\);
\item \(\mathcal{X} \cap \{P_0, P_1,P_2\} = \{P_0, P_1,P_2\}\).
\end{itemize}
\begin{thm}\label{casev,thm}
Let $\mathcal{X}\subset \mathbb{P}^{n+1}$ be a smooth hypersurface of degree $d$.
We assume that $\mathcal{X}$ admits an automorphism $g:=[D(a,a,c,I_{n-1})]$ and $W(X_0,X_1,X_2)\subset \mathcal X$ where $a,c\in k^{\times}\setminus\{1\}$ with $a\neq c$.
Then $n\leq 3$.   If $n=3$, then the order $\mr{ord}(g)$ of $g$ divides one of the integers
\[
(d-1)d,\quad (d-1)^2,\quad (d-2)d,\quad d^2-3d+3,\quad\mr{or}\quad (d-2)(d-1).
\]
If $n=2$, then $\mr{ord}(g)$ divides one of the integers
\[
(d-1)d,\quad (d-1)^2,\quad\mr{or}\quad (d-2)d. 
\]
Moreover, if $\operatorname{ord}(g) = ld$ or $l(d-1)$ for $l\geq2$, then $\mathcal{X} / \langle g \rangle$ is a rational variety.
\end{thm}
\begin{proof}
By Lemma~\ref{casev,1}, we have $n \leq 3$.
First, we assume that $n = 2$. In this case,
\[
g = [D(a,a,c,1)] = [D(I_2, ca^{-1}, a^{-1})],
\]
and hence the situation is reduced to Case~(iii).
Therefore, it follows from Theorem~\ref{caseiii,thm} that Theorem~\ref{casev,thm} holds.
Consequently, in what follows, we may assume that $n = 3$.

First, we assume that $\mathcal{X} \cap \{P_0, P_1,P_2\} = \emptyset$. 
Then $F(X_0,\ldots, X_{n+1})$ contains the monomials $X_0^{d}$, $X_1^{d}$, and $X_2^{d}$. 
By the equation $\eqref{casesiv,v,vi,u}$ and Lemma~\ref{casev,1},
\[
t=a^{d}=c^{d}=a.
\]
We obtain
$a^{d-1}=1$ and $c^{d}=a$.
By substituting \(a=c^{d}\) into \(a^{d-1}=1\), we get $c^{(d-1)d}=1$.
Since
\[
a=c^{d}\quad\mr{and}\quad c^{(d-1)d}=1,
\]
$\mr{ord}(g)\ \mr{divides}\ (d-1)d$.
In addition, if $\operatorname{ord}(g) =l(d-1)$ for $l\geq2$, then $g^l=[D(c^l,c^l,c^l,I_2)]$.
Since $P_0\not\in\mac X$, $W(X_3,X_4)\not\subset \mac X$.
By Theorem \ref{back}, $\mathcal{X} / \langle g^l \rangle$ is rational, and hence $X/\langle g\rangle$ is rational.
If $\operatorname{ord}(g) =ld$ for $l\geq2$, then $g^l=[D(I_2,c^l,I_2)]$.
Theorem~\ref{r1} implies that $\mathcal{X} / \langle g^l \rangle$ is rational, and hence $\mac X/\langle g\rangle$ is rational.
\\

Second, we assume that $\mathcal{X} \cap \{P_0, P_1,P_2\} = \{P_2\}$. 
Then $F(X_0,\ldots, X_{n+1})$ contains the monomials $X_0^{d}$, $X_1^{d}$, and $X_iX_2^{d-1}$ for some $i\neq2$.

We assume that $i\in\{0,1\}$.
By the equation $\eqref{casesiv,v,vi,u}$ and Lemma~\ref{casev,1},
\[
t=a^{d}=ac^{d-1}=a.
\]
Since
\[
a^{d-1}=c^{d-1}=1,
\]
$\mr{ord}(g)$ divides $d-1$.

We assume that $i\not\in\{0,1\}$.
By the equation $\eqref{casesiv,v,vi,u}$ and Lemma~\ref{casev,1},
\[
t=a^{d}=c^{d-1}=a.
\]
We obtain
$a^{d-1}=1$ and $c^{d-1}=a$.
By substituting \(a=c^{d-1}\) into \(a^{d-1}=1\), we get $c^{(d-1)^2}=1$.
Since
\[
a=c^{d-1}\quad\mr{and}\quad c^{(d-1)^2}=1,
\]
$\mr{ord}(g)\ \mr{divides}\ (d-1)^2$.
In addition, if $\operatorname{ord}(g) =l(d-1)$ for $l\geq2$, then $g^l=[D(I_2,c^l,I_2)]$.
Theorem~\ref{r1} implies that $\mathcal{X} / \langle g^l \rangle$ is rational, and hence $\mac X/\langle g\rangle$ is rational.
\\

Third, we assume that $\mathcal{X} \cap \{P_0, P_1,P_2\} = \{P_1\}$. 
Then $F(X_0,\ldots,X_{n+1})$ contains monomials $X_0^{d}$, $X_i X_1^{d-1}$, and $X_2^{d}$ for some $i \neq 1$.
If $i=2$, then the equation $\eqref{casesiv,v,vi,u}$ implies that
\[
t=a^d=ca^{d-1}=c^d.
\]
By $a^d=ca^{d-1}$, we obtain $a=c$. This contradicts $a \neq c$. Thus, $i\neq 2$, and hence $i \notin \{0,2\}$.
By the equation $\eqref{casesiv,v,vi,u}$,
\[
t=a^d=a^{d-1}=c^d.
\]
By $a^d=a^{d-1}$, we get $a = 1$. This contradicts $a \neq 1$.
Therefore, we conclude that $i = 0$.
By the equation $\eqref{casesiv,v,vi,u}$, 
\[
t=a^d=a^d=c^d.
\]
The situation reduces to the case where $\mathcal{X}\cap\{P_0,P_1,P_2\}=\emptyset$.
Hence, $\mr{ord}(g)$ divides $(d-1)d$.
Moreover, if $\mr{ord}(g)=l(d-1)$ or $ld$ for $l\geq2$, then $\mac X/\langle g\rangle$ is rational.
\\

Fourth, we assume that $\mathcal{X} \cap \{P_0, P_1,P_2\} = \{P_1,P_2\}$. 
Then $F(X_0,\ldots,X_{n+1})$ contains monomials $X_0^{d}$, $X_i X_1^{d-1}$, and $X_jX_2^{d-1}$ for some $i \neq 1$ and $j \neq 2$.
In the same way, we have 
\[
i=0.
\]
Then $t=a^d$, and the argument reduces to the case where $\mathcal{X} \cap \{P_0, P_1,P_2\} = \{P_2\}$.
Hence, $\mr{ord}(g)$ divides  $(d-1)^2$.
Moreover, if $\mr{ord}(g)=l(d-1)$ for $l\geq2$, then $\mac X/\langle g\rangle$ is rational.
\\

Fifth, we assume that $\mathcal{X} \cap \{P_0, P_1,P_2\} = \{P_0,P_1\}$. 
Then $F(X_0,\ldots, X_{n+1})$ contains the monomials $X_iX_0^{d-1}$, $X_jX_1^{d-1}$, and $X_2^{d}$ for some $i\neq0$ and $j\neq1$.

We assume that $i=1$.
Since $a \neq c$, 
by the equation $\eqref{casesiv,v,vi,u}$ we have $j \neq 2$.
If $j \notin \{0,2\}$, then by the equation $\eqref{casesiv,v,vi,u}$ we get $a = 1$, which contradicts the assumption $a \neq 1$.
Therefore, we conclude that $j = 0$, and hence
\[
t=a^d.
\]
By the equation $\eqref{casesiv,v,vi,u}$, the situation reduces to the case where $\mathcal{X} \cap \{P_0, P_1,P_2\} =\emptyset$. Hence, $\mr{ord}(g)$ divides $(d-1)d$.
Moreover, if $\mr{ord}(g)=l(d-1)$ or $ld$ for $l\geq2$, then $\mac X/\langle g\rangle$ is rational.

Similarly, we assume that $j\neq0$.
We assume that $i\neq1$.
In addition, we assume that $i=2$.
If $j\neq2$, then
by the equation $\eqref{casesiv,v,vi,u}$ we have $c=1$.
This contradicts that $c\neq1$.
Thus, $j=2$, and hence 
\[
t=ca^{d-1}.
\]
We assume that $k\in\{0,1\}$.
By the equation $\eqref{casesiv,v,vi,u}$ and Lemma~\ref{casev,1},
\[
t=ca^{d-1}=c^{d}=a.
\]
We obtain
$ca^{d-2}=1$ and $c^{d}=a$.
By substituting \(a=c^{d}\) into \(ca^{d-2}=1\), we get $c^{(d-1)^2}=1$.
Since
\[
a=c^d\quad \mr{and}\quad c^{(d-1)^2}=1,
\]
$\mr{and}\ \mr{ord}(g)\ \mr{divides}\ (d-1)^2$.
In addition, if $\operatorname{ord}(g) =l(d-1)$ for $l\geq2$, then $g^l=[D(c^l,c^l,c^l,I_2)]$.
Since $P_0\not\in\mac X$, $W(X_3,X_4)\not\subset \mac X$.
By Theorem \ref{back}, $\mathcal{X} / \langle g^l \rangle$ is rational, and hence $X/\langle g\rangle$ is rational.

We assume that $i\not\in\{0,1\}$.
By the case $i=2$, we get that $j\not\in\{0,2\}$.
By the equation $\eqref{casesiv,v,vi,u}$ and Lemma~\ref{casev,1},
\[
t=a^{d-1}=c^{d}=a.
\]
We obtain
$a^{d-2}=1$ and $c^{d}=a$.
By substituting \(a=c^{d}\) into \(a^{d-2}=1\), we get $c^{(d-2)d}=1$.
Since
\[
a=c^d\quad \mr{and}\quad c^{(d-2)d}=1,
\]
$\mr{ord}(g)\ \mr{divides}\ (d-2)d$.
In addition, if $\operatorname{ord}(g) =ld$ for $l\geq2$, then $g^l=[D(I_2,c^l,I_2)]$.
Theorem~\ref{r1} implies that $\mathcal{X} / \langle g^l \rangle$ is rational, and hence $\mac X/\langle g\rangle$ is rational.
\\

Finally, we assume that $\mathcal X\cap \{P_0,P_1,P_2\}=\{P_0,P_1,P_2\}$. Then 
$F(X_0,\ldots, X_{n+1})$ contains the monomials $X_iX_0^{d-1}$, $X_jX_1^{d-1}$, and $X_kX_2^{d-1}$ for some $i\neq0$, $j\neq1$, and $k\neq 2$. 

We assume that $i=1$.
Then
\[
t=a^d.
\]
Since $a \neq 1$ and $a\neq c$,  we have $j=0$.
Then the situation reduces to the case where $\mathcal X\cap \{P_0,P_1,P_2\}=\{P_2\}$.
Hence, $\mr{ord}(g)$ divides $(d-1)^2$.
Moreover, if $\mr{ord}(g)=l(d-1)$ for $l\geq2$, then $\mac X/\langle g\rangle$ is rational.

Similarly, if $j=0$, then $t=a^d$ and the situation reduces to the case where $\mathcal X\cap \{P_0,P_1,P_2\}=\{P_2\}$.
Hence, $\mr{ord}(g)$ divides $(d-1)^2$.
Moreover, if $\mr{ord}(g)=l(d-1)$ for $l\geq2$, then $\mac X/\langle g\rangle$ is rational.

Now, we assume that $i\neq 1$ and $j\neq0$.
We assume that $i=2$.
Then 
\[
t=ca^{d-1}.
\]
Since $c\neq1$, we have $j=2$.

In addition, we assume that $k\in\{0,1\}$.
By the equation $\eqref{casesiv,v,vi,u}$ and Lemma~\ref{casev,1},
\[
t=ca^{d-1}=ac^{d-1}=a.
\]
We obtain $ca^{d-2}=c^{d-1}=1$.
By substituting \(c=a^{2-d}\) into \(c^{d-1}=1\), we get $a^{(d-2)(d-1)}=1$.
Since
\[
c=a^{2-d}\quad \mr{and}\quad a^{(d-2)(d-1)}=1,
\]
$\mr{ord}(g)\ \mr{divides}\ (d-2)(d-1)$.
In addition, if $\operatorname{ord}(g) =l(d-1)$ for $l\geq2$, then $g^l=[D(I_2,c^l,I_2)]$.
Theorem~\ref{r1} implies that $\mathcal{X} / \langle g^l \rangle$ is rational, and hence $\mathcal X/\langle g\rangle$ is rational.

We assume that $k\not\in\{0,1\}$.
By the equation $\eqref{casesiv,v,vi,u}$ and Lemma~\ref{casev,1},
\[
t=ca^{d-1}=c^{d-1}=a.
\]
We obtain
$ca^{d-2}=1$ and $c^{d-1}=a$.
By substituting \(a=c^{d-1}\) into \(ca^{d-2}=1\), we get $c^{d^2-3d+3}=1$.
Since
\[
a=c^{d-1}\quad\mr{and}\quad c^{d^2-3d+3}=1,
\]
$\mr{ord}(g)\ \mr{divides}\ d^2-3d+3$.

We assume that $i\neq2$.
Since $a \neq 1$ and $c\neq1$, 
by the equation $\eqref{casesiv,v,vi,u}$ we have $j\neq2$.
Then
\[
t=a^{d-1}.
\]

We assume that $k\in\{0,1\}$.
By the equation $\eqref{casesiv,v,vi,u}$ and Lemma~\ref{casev,1},
\[
t=a^{d-1}=ac^{d-1}=a.
\]
Since
\[
a^{d-2}=c^{d-1}=1,
\]
$\mr{ord}(g)\ \mr{divides}\ (d-2)(d-1)$.
In addition, if $\operatorname{ord}(g) =l(d-1)$ for $l\geq2$, then $g^l=[D(I_2,c^l,I_2)]$.
Theorem~\ref{r1} implies that $\mathcal{X} / \langle g^l \rangle$ is rational, and hence $\mac X/\langle g\rangle$ is rational.

We assume that $k\not\in\{0,1\}$.
By the equation $\eqref{casesiv,v,vi,u}$ and Lemma~\ref{casev,1},
\[
t=a^{d-1}=c^{d-1}=a.
\]
We obtain
$a^{d-2}=1$ and $c^{d-1}=a$.
By substituting \(a=c^{d-1}\) into \(a^{d-2}=1\), we get $c^{(d-2)(d-1)}=1$.
Since
\[
a=c^{d-1}\quad\mr{and}\quad c^{(d-2)(d-1)}=1,
\]
$\mr{ord}(g)\ \mr{divides}\ (d-2)(d-1)$.
In addition, if $\operatorname{ord}(g) =l(d-1)$ for $l\geq2$, then $g^l=[D(I_2,c^l,I_2)]$.
Theorem~\ref{r1} implies that $\mathcal{X} / \langle g^l \rangle$ is rational, and hence $\mac X/\langle g\rangle$ is rational.
\end{proof}

In [\ref{bio:th21l}], the following result is stated.
\begin{thm}\label{r3}
$([\ref{bio:th21l},\ \mr{Theorem}\ 1.10])$.
Let $\mathcal X \subset \mathbb P^{n+1}$ be a smooth hypersurface of degree
$d \ge 4$ with $n \ge 2$, and let
$g \in \mathrm{PGL}(n+2,k)$ be an automorphism of $\mathcal X$
such that ${\rm ord}(g)=k(d-1)$ for some integer $k \ge 2$.
If one of the following holds, then $\mathcal X$ has a Galois point $P \in \mathcal X$, and
$\langle g^k\rangle$ is the Galois group of the projection
$\mathcal X \dashrightarrow \mathbb P^n$ from the point $P$.
\begin{enumerate}[$(a)$]
\item
$n=2$ and $|{\rm Fix}(g)| \ge 5$;
\item
$n \ge 3$ and ${\rm Fix}(g)$ has codimension two in $\mathcal X$.
\end{enumerate}
\end{thm}
However, Theorem~\ref{r3} requires a correction.
The reason is that, in the proof of Theorem~\ref{r3}, the case where
${\rm Fix}(g)$ has codimension two was not treated exhaustively:
specifically, Case~(v) was overlooked.
More precisely, when $n=3$ and
$\mathcal X \cap \{P_0,P_1,P_2\} = \emptyset$ in Theorem~\ref{casev,thm},
the fixed locus ${\rm Fix}(g)$ has codimension two in $\mathcal X$
and contains a projective line.
In this situation, the result of~[\ref{bio:th21l}] does not apply.
Indeed, such an example actually exists, as shown below.

\begin{exa}
Let $\mathcal X \subset \mathbb P^4$ be the smooth hypersurface defined by
\[
X_0^d + X_1^d + X_2^d + X_0 X_3^{d-1} + X_1 X_4^{d-1} = 0.
\]
Then $\mathcal X$ admits an automorphism
\[
g := [D(\zeta^{d}, \zeta^{d}, \zeta, I_2)]
\]
of order $d(d-1)$, where $\zeta$ denotes a primitive $d(d-1)$-th root of unity.
\end{exa}
We present below a corrected statement of Theorem~\ref{r3} in the case $n=3$.
The result follows directly from Theorem~\ref{casev,thm}.
\begin{thm}\label{r3-corrected}
Let $\mathcal X \subset \mathbb P^4$ be a smooth hypersurface of degree
$d$, and let
$g \in \mathrm{PGL}(5,k)$ be an automorphism of $\mathcal X$
such that ${\rm ord}(g)=k(d-1)$ for some integer $k \ge 2$.
If ${\rm Fix}(g)$ has codimension two in $\mathcal X$,
and ${\rm Fix}(g)$ does not contain a projective line, then $\mathcal X$ has a Galois point $P \in \mathcal X$, and
$\langle g^k\rangle$ is the Galois group of the projection
$\mathcal X \dashrightarrow \mathbb P^3$ from the point $P$.
\end{thm}

\subsection{Type VI}
We consider Type VI.
Throughout this subsection, we assume that $a$, $b$, and $c$ are pairwise distinct elements.
After permuting \(X_0, X_1\), and  \(X_2\) if necessary, 
it suffices to consider the following possible configurations of 
\(\mathcal{X} \cap \{P_0, P_1, P_2\}\):
\begin{itemize}
\item \(\mathcal{X} \cap \{P_0, P_1, P_2\} = \emptyset\);
\item \(\mathcal{X} \cap \{P_0, P_1, P_2\} = \{P_2\}\);
\item \(\mathcal{X} \cap \{P_0, P_1, P_2\} = \{P_1, P_2\}\);
\item \(\mathcal{X} \cap \{P_0, P_1, P_2\} = \{P_0, P_1, P_2\}\).
\end{itemize}
\begin{lem}\label{casevi,1}
Let $\mathcal{X}\subset \mathbb{P}^{n+1}$ be a smooth hypersurface of degree $d$.
We assume that $\mathcal{X}$ admits an automorphism $g:=[D(a,b,c,I_{n-1})]$ and $W(X_0,X_1,X_2)\subset \mathcal X$ where $a,b,c\in k^{\times}\backslash\{1\}$ are pairwise distinct. 
Then $n=2$, and one of the following holds:
\begin{enumerate}[$(i)$]
\item \( t = a \) and 
\[\sum_{u=0,1} F_{u\,_{\underline{0,1,2}}}(X_0,X_1,X_2,X_{3})= F_{1,0,0}(X_3)X_0.
\]
\item \( t = b \) and
\[
\sum_{u=0,1} F_{u\,_{\underline{0,1,2}}}(X_0,X_1,X_2,X_{3})= F_{0,1,0}(X_3)X_1 .
\]
\item \( t = c \) and 
\[\sum_{u=0,1} F_{u\,_{\underline{0,1,2}}}(X_0,X_1,X_2,X_{3})= F_{0,0,1}(X_3)X_2.
\]
\end{enumerate}
\end{lem}
\begin{proof}
This lemma can be proved in the same manner as Lemma~\ref{caseiv,lem1}.
We omit the proof.
\end{proof}
\begin{thm}\label{casevi,thm}
Let $\mathcal{X}\subset \mathbb{P}^{n+1}$ be a smooth hypersurface of degree $d$.
We assume that $\mathcal{X}$ admits an automorphism $g:=[D(a,b,c,I_{n-1})]$ and $W(X_0,X_1,X_2)\subset \mathcal X$ where $a,b,c\in k^{\times}\setminus\{1\}$ are pairwise distinct.
Then the order of $g$ divides one of the integers
\[
(d-1)^2d,\quad (d-1)^3,\quad (d^2-3d+3)d,\quad (d^2-3d+3)(d-1),
\]
\[
(d-2)(d-1)d,\quad \mr{or}\quad (d-2)(d-1)^2.
\]
\end{thm}
\begin{proof}
First, we assume that \(\mathcal{X} \cap \{P_0, P_1, P_2\} = \emptyset\).
Then $F(X_0,\ldots, X_{n+1})$ contains the monomials $X_0^{d}$, $X_1^{d}$, and $X_2^{d}$. 
By the equation $\eqref{casesiv,v,vi,u}$, we have 
\[
t=a^{d}=b^{d}=c^{d}.
\]
By Lemma~\ref{casevi,1},
\[
a^{d}=b^{d}=c^{d}=
\left\{
\begin{aligned}
&a\quad(\mr{case}~(i)\ \mr{of\ Lemma}~\ref{casevi,1}),\\
&b\quad(\mr{case}~(ii)\ \mr{of\ Lemma}~\ref{casevi,1}),\\
&c\quad(\mr{case}~(iii)\ \mr{of\ Lemma}~\ref{casevi,1}).
\end{aligned}
\right.
\]
In case $(i)$, we have $a^{d-1}=1$, and $a=b^{d}=c^{d}$.
By substituting $a=b^{d}$ (resp. $a=c^{d}$) into \(a^{d-1}=1\), we get $b^{(d-1)d}=1$ (resp. $c^{(d-1)d}=1$).
Since
\[
a=b^d=c^d,\quad \mr{and}\quad b^{(d-1)d}=c^{(d-1)d}=1,
\]
$\mr{ord}(g)\ \mr{divides}\ (d-1)d$.
Cases~$(ii)$ and~$(iii)$ are reduced to case $(i)$
by exchanging $X_0$, $X_1$, and $X_2$.
Thus, $\mr{ord}(g)\ \mr{divides}\ (d-1)d$.
\\

Second, we assume that \(\mathcal{X} \cap \{P_0, P_1, P_2\} = \{P_2\}\).
Then $F(X_0,\ldots, X_{n+1})$ contains the monomials $X_0^{d}$, $X_1^{d}$, and $X_iX_2^{d-1}$ for some $i\neq2$. 

We assume that $i=0$.
By the equation $\eqref{casesiv,v,vi,u}$, we have 
\[
t=a^{d}=b^{d}=ac^{d-1}.
\]
By Lemma~\ref{casevi,1},
\[
a^{d}=b^{d}=ac^{d-1}=
\left\{
\begin{aligned}
&a\quad(\mr{case}~(i)\ \mr{of\ Lemma}~\ref{casevi,1}),\\
&b\quad(\mr{case}~(ii)\ \mr{of\ Lemma}~\ref{casevi,1}),\\
&c\quad(\mr{case}~(iii)\ \mr{of\ Lemma}~\ref{casevi,1}).
\end{aligned}
\right.
\]
In case $(i)$, we have $a^{d-1}=c^{d-1}=1$ and $b^{d}=a$.
By substituting $a=b^{d}$ into \(a^{d-1}=1\), we get $b^{(d-1)d}=1$.
Since
\[a=b^d,\quad b^{(d-1)d}=1,\quad \mr{and}\quad c^{d-1}=1,\]
$\mr{ord}(g)\ \mr{divides}\ (d-1)d$.
Case~$(ii)$ is reduced to case~$(i)$ by interchanging $X_0$ and~$X_1$.
Hence, $\mr{ord}(g)$ divides $(d-1)d$.
In case $(iii)$, we get $a^{d}=b^{d}=c$ and $a=c^{2-d}$.
By substituting \(a=c^{2-d}\) into \(a^{d}=c\), we obtain $c^{(d-1)^2}=1$.
By substituting \(a^d=c\) into \(a=c^{2-d}\), we obtain $a^{(d-1)^2}=1$.
Since 
\[a=b^{(2-d)d},\quad
b^{(d-1)^2d}=1,\quad\mr{and}\quad
c=b^{d},\]
$\mr{ord}(g)\ \mr{divides}\ (d-1)^2d$.

The case $i=1$ is reduced to the case $i=0$ by interchanging $X_0$ and $X_1$.
Hence, $\mr{ord}(g)$ divides $(d-1)^2d$.

We assume that $i\not\in\{0,1\}$.
By the equation $\eqref{casesiv,v,vi,u}$, we have 
\[
t=a^{d}=b^{d}=c^{d-1}.
\]
By Lemma~\ref{casevi,1},
\[
a^{d}=b^{d}=c^{d-1}=
\left\{
\begin{aligned}
&a\quad(\mr{case}~(i)\ \mr{of\ Lemma}~\ref{casevi,1}),\\
&b\quad(\mr{case}~(ii)\ \mr{of\ Lemma}~\ref{casevi,1}),\\
&c\quad(\mr{case}~(iii)\ \mr{of\ Lemma}~\ref{casevi,1}).
\end{aligned}
\right.
\]
In case $(i)$, we have $a^{d-1}=1$ and $b^{d}=c^{d-1}=a$.
By substituting $a=b^d$ into $a^{d-1}=1$, we get $b^{(d-1)d}=1$.
By substituting $a=c^{d-1}$ into $a^{d-1}=1$, we get $c^{(d-1)^2}=1$. 
Since
\[
a=b^d=c^{d-1}\quad \mr{and}\quad b^{(d-1)d}=c^{(d-1)^2}=1,
\]
$\mr{ord}(g)\ \mr{divides}\ (d-1)^2d$.
Case~$(ii)$ is reduced to case~$(i)$ by interchanging $X_0$ and~$X_1$.
Hence, $\mr{ord}(g)$ divides $(d-1)^2d$.
In case $(iii)$, we have 
$a^{d}=b^{d}=c$ and $c^{d-2}=1$.
Since
\[
a^{(d-2)d}=b^{(d-2)d}=1\quad \mr{and}\quad a^d=b^d=c,
\]
$\mr{ord}(g)\ \mr{divides}\ (d-2)d$.
\\

Third, we assume that \(\mathcal{X} \cap \{P_0, P_1, P_2\} = \{P_1, P_2\}\).
Then $F(X_0,\ldots, X_{n+1})$ contains the monomials $X_0^{d}$, $X_iX_1^{d-1}$, and $X_jX_2^{d-1}$ for some $i\neq1$ and some $j\neq2$. 

We assume that $i=0$ and $j=0$. 
By the equation $\eqref{casesiv,v,vi,u}$, we have 
\[
t=a^{d}=ab^{d-1}=ac^{d-1}.
\]
By Lemma~\ref{casevi,1},
\[
a^{d}=ab^{d-1}=ac^{d-1}=
\left\{
\begin{aligned}
&a\quad(\mr{case}~(i)\ \mr{of\ Lemma}~\ref{casevi,1}),\\
&b\quad(\mr{case}~(ii)\ \mr{of\ Lemma}~\ref{casevi,1}),\\
&c\quad(\mr{case}~(iii)\ \mr{of\ Lemma}~\ref{casevi,1}).
\end{aligned}
\right.
\]
In case $(i)$, we have 
\[
a^{d-1}=b^{d-1}=c^{d-1}=1.
\]
Thus, $\mr{ord}(g)\ \mr{divides}\ d-1$.
In case $(ii)$, we have $a^{d}=ac^{d-1}=b$ and $ab^{d-2}=1$.
Since
\[
a=b^{2-d},\quad b^{d-1}=c^{d-1},\quad \mr{and}\quad b^{(d-1)^2}=1,
\]
$\mr{ord}(g)\ \mr{divides}\ (d-1)^2$.
Case~$(iii)$ is reduced to case~$(ii)$ by interchanging $X_1$ and~$X_2$.
Hence, $\mr{ord}(g)$ divides $(d-1)^2$.

We assume that $i=0$ and $j=1$. 
By the equation $\eqref{casesiv,v,vi,u}$, we have 
\[
t=a^{d}=ab^{d-1}=bc^{d-1}.
\]
By Lemma~\ref{casevi,1},
\[
a^{d}=ab^{d-1}=bc^{d-1}=
\left\{
\begin{aligned}
&a\quad(\mr{case}~(i)\ \mr{of\ Lemma}~\ref{casevi,1}),\\
&b\quad(\mr{case}~(ii)\ \mr{of\ Lemma}~\ref{casevi,1}),\\
&c\quad(\mr{case}~(iii)\ \mr{of\ Lemma}~\ref{casevi,1}).
\end{aligned}
\right.
\]
In case $(i)$, we have $a^{d-1}=b^{d-1}=1$ and $c^{d-1}=a$.
By substituting $c^{d-1}=a$ into $a^{d-1}=1$, we get $c^{(d-1)^2}=1$.
Since
\[
a^{d-1}=b^{d-1}=c^{(d-1)^2}=1\quad \mr{and}\quad c^{d-1}=a,
\]
$\mr{ord}(g)\ \mr{divides}\ (d-1)^2$.
In case $(ii)$, we obtain $a^d=b$ and $ab^{d-2}=c^{d-1}=1$.
By substituting $a^d=b$ into $ab^{d-2}=1$, we get $a^{(d-1)^2}=1$.
Since
\[
a^d=b\quad\mr{and}\quad a^{(d-1)^2}=c^{d-1}=1,
\]
$\mr{ord}(g)\ \mr{divides}\ (d-1)^2$.
In case $(iii)$,
we have $a^{d}=ab^{d-1}=c$ and $bc^{d-2}=1$.
By substituting \(a^{d}=c\) into \(bc^{d-2}=1\), we obtain $b = a^{-(d-2)d}$.
By substituting \(a^{d}=c\) and \(b=a^{-(d-2)d}\) into \(ab^{d-1}=c\), we obtain $a^{(d-1)^3}=1$. Since
\[
a^{(d-1)^3}=1,\quad b=a^{-(d-2)d},\quad \mr{and}\quad c=a^d,
\]
$\mr{ord}(g)\ \mr{divides}\ (d-1)^3$.

We assume that $i=0$ and $j\not\in\{0,1,\}$. 
By the equation $\eqref{casesiv,v,vi,u}$, we have 
\[
t=a^{d}=ab^{d-1}=c^{d-1}.
\]
By Lemma~\ref{casevi,1},
\[
a^{d}=ab^{d-1}=c^{d-1}=
\left\{
\begin{aligned}
&a\quad(\mr{case}~(i)\ \mr{of\ Lemma}~\ref{casevi,1}),\\
&b\quad(\mr{case}~(ii)\ \mr{of\ Lemma}~\ref{casevi,1}),\\
&c\quad(\mr{case}~(iii)\ \mr{of\ Lemma}~\ref{casevi,1}).
\end{aligned}
\right.
\]
In case $(i)$, we have
$a^{d-1}=b^{d-1}=1\quad\mr{and}\quad c^{d-1}=a$.
By substituting $c^{d-1}=a$ into $a^{d-1}=1$,
we get $c^{(d-1)^2}=1$.
Since
\[
a=c^{d-1}\quad \mr{and}\quad c^{(d-1)^2}=b^{d-1}=1,
\]
$\mr{ord}(g)\ \mr{divides}\ (d-1)^2$.
In case $(ii)$, we obtain
$a^d=c^{d-1}=b$ and $ab^{d-2}=1$.
By substituting \(b=c^{d-1}\) into \(ab^{d-2}=1\), we have $a= c^{-(d-2)(d-1)}$.
By substituting \(a=c^{-(d-2)(d-1)}\) and \(b=c^{d-1}\) into
\(a^{d}=b\), we obtain $c^{(d-1)^3}=1$.
Since
\[
a=c^{-(d-2)(d-1)},\quad b=c^{d-1},\quad\mr{and} \quad c^{(d-1)^3}=1,
\]
$\mr{ord}(g)\ \mr{divides}\ (d-1)^3$.
In case $(iii)$,
we have $a^{d}=ab^{d-1}=c$ and $c^{d-2}=1$.
By substituting \(a^{d}=c\) into \((ab^{d-1})^d=c^d\), we get $cb^{(d-1)d}=c^d$.
By $c^{d-2}=1$, we have $c=b^{(d-1)d}$.
By substituting \(c=b^{(d-1)d}\)  into \(ab^{d-1}=c\), we obtain $a=b^{(d-1)^2}$.
Moreover, by substituting \(c=b^{(d-1)d}\) into $c^{d-2}=1$, we get $b^{(d-2)(d-1)d}=1$.
Since
\[
a=b^{(d-1)^2},\quad c=b^{(d-1)d},\quad\mr{and}\quad b^{(d-2)(d-1)d}=1,\ 
\]
$\mr{ord}(g)\ \mr{divides}\ (d-2)(d-1)d$.

The case $(i,j)=(2,0)$ is reduced to the case $(i,j)=(0,1)$ by interchanging $X_1$ and $X_2$.
Hence, $\mr{ord}(g)$ divides $(d-1)^2d$ or $(d-1)^3$.

We assume that $i=2$ and $j=1$. 
By the equation $\eqref{casesiv,v,vi,u}$, we have 
\[
t=a^{d}=cb^{d-1}=bc^{d-1}.
\]
By Lemma~\ref{casevi,1},
\[
a^{d}=cb^{d-1}=bc^{d-1}=
\left\{
\begin{aligned}
&a\quad(\mr{case}~(i)\ \mr{of\ Lemma}~\ref{casevi,1}),\\
&b\quad(\mr{case}~(ii)\ \mr{of\ Lemma}~\ref{casevi,1}),\\
&c\quad(\mr{case}~(iii)\ \mr{of\ Lemma}~\ref{casevi,1}).
\end{aligned}
\right.
\]
In case $(i)$, we have 
$a^{d-1}=1$ and $cb^{d-1}=bc^{d-1}=a$.
By raising the equalities \(cb^{d-1}=a\) and \(bc^{d-1}=a\) to the \((d-1)\)-st power and using \(a^{d-1}=1\) yields
\[
c^{d-1} b^{(d-1)^{2}} = c^{(d-1)^{2}} b^{d-1} = 1.
\]
Hence
\[
b^{(d-1)^{3}} = c^{(d-1)^{3}} = 1.
\]
On the other hand, from \(cb^{d-1}=bc^{d-1}\) we obtain \(b^{d-2}=c^{d-2}\). 
Since \(\gcd(d-2,(d-1)^{3})=1\), there exist integers \(u,v\) such that
\[
u(d-2) + v(d-1)^{3} = 1.
\]
Therefore, using \(b^{\,(d-1)^{3}}=c^{\,(d-1)^{3}}=1\) and \(b^{\,d-2}=c^{\,d-2}\), we deduce
\[
b = c.
\]
This contradicts the assumption \(b \ne c\). Hence case~$(i)$ cannot occur.
In case $(ii)$, we obtain
$a^d=b$ and $cb^{d-2}=c^{d-1}=1$.
By substituting $c=b^{2-d}$ into $c^{d-1}=1$, we have $b^{(d-1)(d-2)}=1$.
By substituting $b=a^{d}$ into $b^{(d-1)(d-2)}=1$, we get $a^{d(d-1)(d-2)}=1$.
Since
\[
a^{(d-2)(d-1)d}=1,\quad b=a^{d},\quad \mr{and}\quad c=a^{-(d-2)d},
\]
$\mr{ord}(g)\ \mr{divides}\ d(d-1)(d-2)$.
Case~$(iii)$ is reduced to case~$(ii)$ by interchanging \(X_{1}\) and \(X_{2}\).
Hence, $\mr{ord}(g)$ divides $d(d-1)(d-2)$.

We assume that $i=2$ and $j\not\in\{0,1\}$. 
By the equation $\eqref{casesiv,v,vi,u}$, we have 
\[
t=a^{d}=cb^{d-1}=c^{d-1}.
\]
By Lemma~\ref{casevi,1},
\[
a^{d}=cb^{d-1}=c^{d-1}=
\left\{
\begin{aligned}
&a\quad(\mr{case}~(i)\ \mr{of\ Lemma}~\ref{casevi,1}),\\
&b\quad(\mr{case}~(ii)\ \mr{of\ Lemma}~\ref{casevi,1}),\\
&c\quad(\mr{case}~(iii)\ \mr{of\ Lemma}~\ref{casevi,1}).
\end{aligned}
\right.
\]
In case $(i)$, we have 
$a^{d-1}=1$ and $cb^{d-1}=c^{d-1}=a$.
By substituting $c^{d-1}=a$ into $a^{d-1}=1$, we get $c^{(d-1)^2}=1$.
Since $a^{d-1}=c^{(d-1)^2}=1$ and $cb^{d-1}=a$, we obtain $b^{(d-1)^3}=1$. Since
\[
a=b^{-(d-1)^2},\quad b^{(d-1)^3}=1,\quad \mr{and}\quad c=b^{-(d-1)d},
\]
$\mr{ord}(g)\ \mr{divides}\ (d-1)^3$.
In case $(ii)$, we obtain
$a^d=c^{d-1}=b$ and $cb^{d-2}=1$.
By substituting \(a^{d}=b\) into \(cb^{d-2}=1\), we get $c=a^{-(d-2)d}$.
By substituting \(b=a^{d}\) and \(c=a^{-(d-2)d}\) into \(c^{d-1}=b\), we obtain $a^{d(d^2-3d+3)}=1$.
Since
\[
a^{d(d^2-3d+3)}=1,\quad b=a^{d},\quad \mr{and}\quad c=a^{-(d-2)d},
\]
$\mr{ord}(g)\ \mr{divides}\ d(d^2-3d+3)$.
In case $(iii)$,
we have $a^{d}=c$ and $b^{d-1}=c^{d-2}=1$.
By substituting $a^d=c$ into $c^{d-2}=1$, we obtain $a^{d(d-2)}=1$.
Since
\[
a^{(d-2)d}=b^{d-1}=1\quad \mr{and}\quad c=a^{d},
\]
$\mr{ord}(g)\ \mr{divides}\ (d-2)(d-1)d$.

The case $i\not\in\{0,2\}$ and $j=0$ is reduced to the case $i=0$ and $j\not\in\{0,1\}$ by interchanging $X_1$ and $X_2$.
Hence, $\mr{ord}(g)$ divides $(d-1)^3$ or $(d-2)(d-1)d$.

The case $i\not\in\{0,2\}$ and $j=1$ is reduced to the case $i=2$ and $j\not\in\{0,1\}$ by interchanging $X_1$ and $X_2$.
Hence, $\mr{ord}(g)$ divides $(d-1)^3$, $(d^2-3d+3)d$, or $(d-2)(d-1)d$.

We assume that $i\not\in\{0,2\}$ and $j\not\in\{0,1\}$. 
By the equation $\eqref{casesiv,v,vi,u}$, we have 
\[
t=a^{d}=b^{d-1}=c^{d-1}.
\]
By Lemma~\ref{casevi,1},
\[
a^{d}=b^{d-1}=c^{d-1}=
\left\{
\begin{aligned}
&a\quad(\mr{case}~(i)\ \mr{of\ Lemma}~\ref{casevi,1}),\\
&b\quad(\mr{case}~(ii)\ \mr{of\ Lemma}~\ref{casevi,1}),\\
&c\quad(\mr{case}~(iii)\ \mr{of\ Lemma}~\ref{casevi,1}).
\end{aligned}
\right.
\]
In case $(i)$, we have $a^{d-1}=1$ and $b^{d-1}=c^{d-1}=a$.
By substituting $b^{d-1}=a$ (resp. $c^{d-1}=a$) into $a^{d-1}=1$,
we have $b^{(d-1)^2}=1$ (resp. $c^{(d-1)^2}=1$).
Since
\[
a=b^{d-1}=c^{d-1}\quad \mr{and}\quad b^{(d-1)^2}=c^{(d-1)^2}=1,
\]
$\mr{ord}(g)\ \mr{divides}\ (d-1)^2$.
In case $(ii)$, we obtain
$a^d=c^{d-1}=b$ and $b^{d-2}=1$.
By substituting $a^{d}=b$ (resp. $c^{d-1}=b$) into $b^{d-2}=1$,
we have $a^{d(d-2)}=1$ (resp. $c^{(d-1)(d-2)}=1$).
Since
\[
a^{(d-2)d}=c^{(d-1)(d-2)}=1\quad \mr{and}\quad b=a^{d}=c^{d-1},
\]
$\mr{ord}(g)\ \mr{divides}\ (d-2)(d-1)d$.
Case~$(iii)$ is reduced to case~$(ii)$ of the same lemma by interchanging \(X_{1}\) and \(X_{2}\).
Hence, $\mr{ord}(g)$ divides $(d-2)(d-1)d$.
\\

Finally, we assume that \(\mathcal{X} \cap \{P_{0},P_{1},P_{2}\}=\{P_{0},P_{1},P_{2}\}\).
Then \(F(X_{0},\ldots,X_{n+1})\) contains the monomials
\(X_{i}X_{0}^{d-1}\), \(X_{j}X_{1}^{d-1}\), and \(X_{k}X_{2}^{d-1}\)
for some \(i\neq 0\), \(j\neq 1\), and \(k\neq 2\).
If necessary, by permuting \(X_{0},X_{1},X_{2}\), 
we may divide the discussion into the following two cases:
(I) \(X_{i}=X_{1}\) or (II) \(i\notin\{1,2\}\), \(j\notin\{0,2\}\), and \(k\notin\{0,1\}\).

We assume that $i=1$.
We assume that $j=k=0$.
By the equation $\eqref{casesiv,v,vi,u}$, we have 
\[
t=ba^{d-1}=ab^{d-1}=ac^{d-1}.
\]
By Lemma~\ref{casevi,1},
\[
ba^{d-1}=ab^{d-1}=ac^{d-1}=
\left\{
\begin{aligned}
&a\quad(\mr{case}~(i)\ \mr{of\ Lemma}~\ref{casevi,1}),\\
&b\quad(\mr{case}~(ii)\ \mr{of\ Lemma}~\ref{casevi,1}),\\
&c\quad(\mr{case}~(iii)\ \mr{of\ Lemma}~\ref{casevi,1}).
\end{aligned}
\right.
\]
In case $(i)$, we have $ba^{d-2}=b^{d-1}=c^{d-1}=1$.
By substituting \(ba^{d-2}=1\) into \(b^{d-1}=1\), we get $a^{(d-2)(d-1)}=1$.
Since
\[
a^{(d-2)(d-1)}=c^{d-1}=1\quad \mr{and}\quad b=a^{2-d},
\]
$\mr{ord}(g)\ \mr{divides}\ (d-2)(d-1)$.
In case $(ii)$, we obtain
$a^{d-1}=ab^{d-2}=1$ and $ac^{d-1}=b$.
By substituting \(ab^{d-2}=1\) into \(a^{d-1}=1\), we get $b^{(d-2)(d-1)}=1$.
By substituting \(ab^{d-2}=1\) into \(ac^{d-1}=b\), we get $b^{d-1}=c^{d-1}$.
Since
\[
a=b^{2-d},\quad b^{(d-2)(d-1)}=1,\quad \mr{and}\quad b^{d-1}=c^{d-1},
\]
$\mr{ord}(g)\ \mr{divides}\ (d-2)(d-1)$.
In case $(iii)$, we obtain
$ba^{d-1}=ab^{d-1}=c$ and $ac^{d-2}=1$.
By substituting \(ac^{d-2}=1\) into \(ba^{d-1}=c\), we get $b=c^{d^2-3d+3}$.
By substituting \(ac^{d-2}=1\) and \(b=c^{d^2-3d+3}\) into $ab^{d-1}=c$, we get $c^{(d-2)(d-1)^2}=1$.
Then 
\[
a=c^{2-d},\quad b=c^{d^2-3d+3},\quad c^{(d-2)(d-1)^2}=1,
\]
$\mr{ord}(g)\ \mr{divides}\ (d-2)(d-1)^2$.

The case $j=0$ and $k=1$ is reduced to the case $j=0$ and $k=0$ by interchanging $X_0$ and $X_1$.
Hence, $\mr{ord}(g)$ divides $(d-2)(d-1)^2$.

We assume that $j=0$ and $k\not\in\{0,1\}$.
By the equation $\eqref{casesiv,v,vi,u}$, we have 
\[
t=ba^{d-1}=ab^{d-1}=c^{d-1}.
\]
By Lemma~\ref{casevi,1},
\[
ba^{d-1}=ab^{d-1}=c^{d-1}=
\left\{
\begin{aligned}
&a\quad(\mr{case}~(i)\ \mr{of\ Lemma}~\ref{casevi,1}),\\
&b\quad(\mr{case}~(ii)\ \mr{of\ Lemma}~\ref{casevi,1}),\\
&c\quad(\mr{case}~(iii)\ \mr{of\ Lemma}~\ref{casevi,1}).
\end{aligned}
\right.
\]
In case $(i)$, we have 
$ba^{d-2}=b^{d-1}=1$ and $c^{d-1}=a$.
By substituting \(c^{d-1}=a\) into \(ba^{d-2}=1\), we get $b=c^{-(d-2)(d-1)}$.
By substituting \(b=c^{-(d-2)(d-1)}\) into \(b^{d-1}=1\), we get $c^{(d-2)(d-1)^2}=1$.
Since 
\[
a=c^{d-1},\quad b=c^{-(d-2)(d-1)},\quad  \mr{and}\quad c^{(d-2)(d-1)^2}=1,
\]
$\mr{ord}(g)\ \mr{divides}\ (d-2)(d-1)^2$.
Case~$(ii)$ is reduced to case~$(i)$ by interchanging \(X_{0}\) and \(X_{1}\).
Hence, $\mr{ord}(g)$ divides $(d-2)(d-1)^2$.
In case $(iii)$, we obtain
$ba^{d-1}=ab^{d-1}=c$ and $c^{d-2}=1$.
Since $ba^{d-1}=ab^{d-1}$, we have $a^{d-2}=b^{d-2}$.
By substituting $ba^{d-1}=c$, i.e. $b=a^{1-d}c$ into $ab^{d-1}=c$, we obtain $a^{-(d-2)d}c^{d-1}=c$.
Since $c^{d-2}=1$, we get $a^{(d-2)d}=1$.
Since $a^{d-2}=b^{d-2}$, we get $b^{(d-2)d}=1$.
Since
\[
a^{(d-2)d}=b^{(d-2)d}=c^{d-2}=1,
\]
$\mr{ord}(g)\ \mr{divides}\ (d-2)d$.

We assume that $j=2$ and $k=0$.
By the equation $\eqref{casesiv,v,vi,u}$, we have 
\[
t=ba^{d-1}=cb^{d-1}=ac^{d-1}.
\]
By Lemma~\ref{casevi,1},
\[
ba^{d-1}=cb^{d-1}=ac^{d-1}=
\left\{
\begin{aligned}
&a\quad(\mr{case}~(i)\ \mr{of\ Lemma}~\ref{casevi,1}),\\
&b\quad(\mr{case}~(ii)\ \mr{of\ Lemma}~\ref{casevi,1}),\\
&c\quad(\mr{case}~(iii)\ \mr{of\ Lemma}~\ref{casevi,1}).
\end{aligned}
\right.
\]
In case $(i)$, we have 
$ba^{d-2}=c^{d-1}=1$ and $cb^{d-1}=a$.
By substituting \(ba^{d-2}=1\), i.e. $b=a^{2-d}$ into \(cb^{d-1}=a\), we get $c=a^{d^2-3d+3}$.
By substituting \(c=a^{d^2-3d+3}\) into \(c^{d-1}=1\), we get $a^{(d-1)(d^2-3d+3)}=1$.
Then 
\[
a^{(d-1)(d^2-3d+3)}=1,\quad b=a^{2-d},\quad  c=a^{d^2-3d+3},
\]
$\mr{ord}(g)\ \mr{divides}\ (d^2-3d+3)(d-1)$.
Case~$(ii)$ (resp. $(iii)$) is reduced to case~$(i)$ by applying the cyclic permutation $(X_0, X_1, X_2) \mapsto (X_1, X_2, X_0)$ (resp. $(X_2,X_0,X_1)$).
Hence, $\mr{ord}(g)$ divides $(d^2-3d+3)(d-1)$.

The case where $(j,k)=(2,1)$ is reduced to the case where $(j,k)=(0,0)$ by applying the cyclic permutation $(X_0, X_1, X_2) \mapsto (X_2, X_0, X_1)$.
Hence, $\mr{ord}(g)$ divides $(d-2)(d-1)^2$.

We assume that $j=2$ and $k\not\in\{0,1\}$.
By the equation $\eqref{casesiv,v,vi,u}$, we have 
\[
t=ba^{d-1}=cb^{d-1}=c^{d-1}.
\]
By Lemma~\ref{casevi,1},
\[
ba^{d-1}=cb^{d-1}=c^{d-1}=
\left\{
\begin{aligned}
&a\quad(\mr{case}~(i)\ \mr{of\ Lemma}~\ref{casevi,1}),\\
&b\quad(\mr{case}~(ii)\ \mr{of\ Lemma}~\ref{casevi,1}),\\
&c\quad(\mr{case}~(iii)\ \mr{of\ Lemma}~\ref{casevi,1}).
\end{aligned}
\right.
\]
In case $(i)$, we have 
$ba^{d-2}=1$ and $cb^{d-1}=c^{d-1}=a$.
By substituting \(a=c^{d-1}\) into \(ba^{d-2}=1\), we get $b=c^{-(d-2)(d-1)}$.
By substituting \(a=c^{d-1}\) and $b=c^{-(d-2)(d-1)}$ into \(cb^{d-1}=a\), we get $c^{(d-1)(d^2-3d+3)}=1$.
Since
\[
a=c^{d-1},\quad b=c^{-(d-2)(d-1)},\quad \mr{and}\quad c^{(d-1)(d^2-3d+3)}=1,
\]
$\mr{ord}(g)\ \mr{divides}\ (d-1)(d^2-3d+3)$.
In case $(ii)$, we have 
$a^{d-1}=cb^{d-2}=1$ and $c^{d-1}=b$.
By substituting \(c=b^{2-d}\) into \(b=c^{d-1}\), we get $b^{d^2-3d+3}=1$.
Then 
\[
a^{d-1}=b^{d^2-3d+3}=1,\quad c=b^{2-d},\ \mr{and}\ \mr{ord}(g)\ \mr{divides}\ (d-1)(d^2-3d+3).
\]
In case $(iii)$, we have $ba^{d-1}=c$ and $b^{d-1}=c^{d-2}=1$.
By substituting $b^{d-1}=c^{d-2}=1$ into $b^{(d-2)(d-1)}a^{(d-2)(d-1)^2}=c^{(d-2)(d-1)}$, we have $a^{(d-2)(d-1)^2}=1$.
Since
\[
a^{(d-2)(d-1)^2}=b^{d-1}=c^{d-2}=1,
\]
$\mr{ord}(g)\ \mr{divides}\ (d-2)(d-1)^2$.

The case where $j\not\in\{0,2\}$ and $k=0$ is reduced to the case where $j=2$ and $k\not\in\{0,1\}$ by applying the cyclic permutation $(X_0, X_1, X_2) \mapsto (X_1, X_2, X_0)$.
Hence, $\mr{ord}(g)$ divides $(d-2)(d-1)^2$ or $(d-1)(d^2-3d+3)$.

We assume that $j\not\in\{0,2\}$ and $k=1$.
By the equation $\eqref{casesiv,v,vi,u}$, we have 
\[
t=ba^{d-1}=b^{d-1}=bc^{d-1}.
\]
By Lemma~\ref{casevi,1},
\[
ba^{d-1}=b^{d-1}=bc^{d-1}=
\left\{
\begin{aligned}
&a\quad(\mr{case}~(i)\ \mr{of\ Lemma}~\ref{casevi,1}),\\
&b\quad(\mr{case}~(ii)\ \mr{of\ Lemma}~\ref{casevi,1}),\\
&c\quad(\mr{case}~(iii)\ \mr{of\ Lemma}~\ref{casevi,1}).
\end{aligned}
\right.
\]
In case $(i)$, we have 
$ba^{d-2}=1$ and $b^{d-1}=bc^{d-1}=a$.
By substituting \(ba^{d-2}=1\) into \(b^{d-1}=a\), we get $a^{d^2-3d+3}=1$.
Since $ba^{d-2}=1$ and $\gcd(d-2,d^2-3d+3)=1$, we have $b^{d^2-3d+3}=1$.
Since $a^{d^2-3d+3}=b^{d^2-3d+3}=1$ and $bc^{d-1}=a$, we obtain $c^{(d-1)(d^2-3d+3)}=1$.
Since
\[
a^{d^2-3d+3}=b^{d^2-3d+3}=c^{(d-1)(d^2-3d+3)}=1,\]
$\mr{ord}(g)\ \mr{divides}\ (d-1)(d^2-3d+3)$.
In case $(ii)$, we have 
\[
a^{d-1}=b^{d-2}=c^{d-1}=1.
\]
Thus, $\mr{ord}(g)\ \mr{divides}\ (d-2)(d-1)$.
Case~$(iii)$ is reduced to case~$(i)$ by interchanging \(X_{0}\) and \(X_{2}\).
Hence, $\mr{ord}(g)$ divides $(d-1)(d^2-3d+3)$.

We assume that $j\not\in\{0,2\}$ and $k\not\in\{0,1\}$.
By the equation $\eqref{casesiv,v,vi,u}$, we have 
\[
t=ba^{d-1}=b^{d-1}=c^{d-1}.
\]
By Lemma~\ref{casevi,1},
\[
ba^{d-1}=b^{d-1}=c^{d-1}=
\left\{
\begin{aligned}
&a\quad(\mr{case}~(i)\ \mr{of\ Lemma}~\ref{casevi,1}),\\
&b\quad(\mr{case}~(ii)\ \mr{of\ Lemma}~\ref{casevi,1}),\\
&c\quad(\mr{case}~(iii)\ \mr{of\ Lemma}~\ref{casevi,1}).
\end{aligned}
\right.
\]
In case $(i)$, we have $ba^{d-2}=1$ and $b^{d-1}=c^{d-1}=a$.
By substituting \(b^{d-1}=a\) into \(ba^{d-2}=1\), we get $b^{d^2-3d+3}=1$.
Since \(b^{d-1}=a\), we obtain $a^{(d^2-3d+3)(d-1)}$.
Since
\[
a=b^{d-1}=c^{d-1}\quad \mr{and}\quad a^{(d^2-3d+3)(d-1)}=1,
\]
$\mr{ord}(g)\ \mr{divides}\ (d^2-3d+3)(d-1)$.
In case $(ii)$, we have 
$a^{d-1}=b^{d-2}=1$ and $c^{d-1}=b$.
By substituting \(c^{d-1}=b\) into \(b^{d-2}=1\), we get $c^{(d-2)(d-1)}=1$.
Since
\[
a^{d-1}=c^{(d-2)(d-1)}=1\quad \mr{and}\quad b=c^{d-1},
\]
$\mr{ord}(g)\ \mr{divides}\ (d-2)(d-1)$.
In case $(iii)$, we have 
$ba^{d-1}=b^{d-1}=c$ and $c^{d-2}=1$.
By substituting \(b^{d-1}=c\) into \(c^{d-2}=1\), we get $b^{(d-2)(d-1)}=1$.
By substituting \(b^{d-1}=c\) into $ba^{d-1}=c$, we get $a^{d-1}=b^{d-2}$.
Since $b^{(d-2)(d-1)}=1$, we have $a^{(d-1)^2}=1$.
Since
\[
a^{(d-1)^2}=b^{(d-2)(d-1)}=1\quad \mr{and}\quad c=b^{d-1},
\]
$\mr{ord}(g)\ \mr{divides}\ (d-2)(d-1)^2$.

We assume that \(i\notin\{1,2\}\), \(j\notin\{0,2\}\), and \(k\notin\{0,1\}\).
By the equation $\eqref{casesiv,v,vi,u}$, we have 
\[
t=a^{d-1}=b^{d-1}=c^{d-1}.
\]
By Lemma~\ref{casevi,1},
\[
a^{d-1}=b^{d-1}=c^{d-1}=
\left\{
\begin{aligned}
&a\quad(\mr{case}~(i)\ \mr{of\ Lemma}~\ref{casevi,1}),\\
&b\quad(\mr{case}~(ii)\ \mr{of\ Lemma}~\ref{casevi,1}),\\
&c\quad(\mr{case}~(iii)\ \mr{of\ Lemma}~\ref{casevi,1}).
\end{aligned}
\right.
\]
In case $(i)$, we have $a^{d-2}=1$ and $b^{d-1}=c^{d-1}=a$.
By substituting $b^{d-1}=a$ (resp. $c^{d-1}=a$) into $a^{d-2}=1$, we get $b^{(d-2)(d-1)}=1$ (resp. $c^{(d-2)(d-1)}=1$).
\[
a=b^{d-1}=c^{d-1}\quad \mr{and}\quad b^{(d-2)(d-1)}=c^{(d-2)(d-1)}=1,
\]
$\mr{ord}(g)\ \mr{divides}\ (d-2)(d-1)$.
Case~$(ii)$ (resp. $(iii)$) is reduced to case~$(i)$ by interchanging \(X_{0}\) and \(X_{1}\) (resp. $X_0$ and $X_2$).
Hence, $\mr{ord}(g)$ divides $(d-2)(d-1)$.
\end{proof}

\section*{Funding}
This research was supported by JSPS KAKENHI Grant Number 23K12507.

\section*{Declarations}
\textbf{Conflict of interest} The authors declared that they have no conflict of interest.


\begin{thebibliography}{99}
\bibitem{ha}\label{bio:acgh}
E. Arbarello, M. Cornalba, P.A. Griffiths, J. Harris, Geometry of algebraic curves, vol. I. Grundlehren der Mathematischen Wissenschaften 267. Springer, New York (1985).

\bibitem{ha}\label{bio:bb16n}
E. Badr and F. Bars, Non-singular plane curves with an element of ``large'' order in its automorphism group, Int. J. Algebra Comput. 26 (2016), 399--434.

\bibitem{ha}\label{bio:bb16a}
E.~E. Badr and F. Bars~Cortina, Automorphism groups of nonsingular plane curves of degree 5, Comm. Algebra {\bf 44} (2016), no.~10, 4327--4340.




\bibitem{ha}\label{bio:bc25}
E.~E. Badr and F. Bars~Cortina, The stratification by automorphism groups of smooth plane sextic curves, Ann. Mat. Pura Appl. (4) {\bf 204} (2025), no.~5, 2005--2048.

\bibitem{ha}\label{bio:fm14}
S. Fukasawa and K. Miura, Galois points for a plane curve and its dual curve, Rend. Semin. Mat. Univ. Padova {\bf 132} (2014), 61--74.

\bibitem{ha}\label{bio:gl11}
V. Gonz\'alez-Aguilera and A. Liendo, Automorphisms of prime order of smooth cubic $n$-folds, Arch. Math. (Basel) {\bf 97} (2011), no.~1, 25--37.

\bibitem{ha}\label{bio:gl13}
V. Gonz\'alez-Aguilera and A. Liendo, On the order of an automorphism of a smooth hypersurface, Israel J. Math. {\bf 197} (2013), no.~1, 29--49.

\bibitem{ha}\label{bio:th21l}
T. Hayashi, Linear automorphisms of smooth hypersurfaces giving Galois points, Bull. Korean Math. Soc. {\bf 58} (2021), no.~3, 617--635.

\bibitem{ha}\label{bio:th21o}
T. Hayashi, Orders of automorphisms of smooth plane curves for the automorphism groups to be cyclic, Arab. J. Math. (Springer) {\bf 10} (2021), no.~2, 409--422.

\bibitem{ha}\label{bio:th21s}
T. Hayashi, Smooth plane curves with freely acting finite groups, Vietnam J. Math. {\bf 49} (2021), no.~4, 1027--1036.

\bibitem{ha}\label{bio:th23g}
T. Hayashi, Galois covers of the projective line by smooth plane curves of large degree, Beitr. Algebra Geom. {\bf 64} (2023), no.~2, 311--365.

\bibitem{ha}\label{bio:th25g}
T. Hayashi, Galois skew lines of smooth surfaces in $\Bbb{P}^3$, Comm. Algebra {\bf 53} (2025), no.~1, 436--449.

\bibitem{ha}\label{bio:th25a}
T. Hayashi, Abelian automorphism groups of smooth hypersurfaces with smooth quotient, Beitr. Algebra Geom. {\bf 66} (2025), no.~4, 941--983.

\bibitem{ha}\label{bio:hs25l}
T. Hayashi and K. Shimahara, Large orders of automorphisms of smooth curves in $\mathbb P^1\times \mathbb P^1$, preprint arXiv:2512.14948 (2025).

\bibitem{ha}\label{bio:mm63}
H. Matsumura and P. Monsky, On the automorphisms of hypersurfaces, J. Math. Kyoto Univ. 3 (1963/1964), 347$-$361.

\bibitem{ha}\label{bio:my00}
K. Miura and H. Yoshihara, Field theory for function fields of plane quartic curves, J. Algebra, 226 (2000), 283$-$294.

\bibitem{ha}\label{bio:ou19}
K. Oguiso and X. Yu, Automorphism groups of smooth quintic threefolds, Asian J. Math. {\bf 23} (2019), no.~2, 201--256.

\bibitem{ha}\label{bio:wu20}
L. Wei and X. Yu, Automorphism groups of smooth cubic threefolds, J. Math. Soc. Japan {\bf 72} (2020), no.~4, 1327--1343.

\bibitem{ha}\label{bio:yz24}
S. Yang, X. Yu and Z. Zhu, Automorphism groups of cubic fivefolds and fourfolds, J. Lond. Math. Soc. (2) {\bf 110} (2024), no.~4, Paper No. e12997, 35 pp..


\bibitem{ha}\label{bio:y03}
H. Yoshihara, Galois points for smooth hypersurfaces, J. Algebra {\bf 264} (2003), no.~2, 520--534.

\bibitem{ha}\label{bio:z22}
Z. Zheng, On abelian automorphism groups of hypersurfaces, Israel J. Math. {\bf 247} (2022), no.~1, 479--498.

\end{thebibliography}
\end{document}